\documentclass[12pt]{amsart}
\usepackage{amssymb}
\usepackage{mathrsfs}
\usepackage{a4}
\usepackage{graphicx}

\usepackage{enumitem}
\usepackage{xcolor}
\usepackage[all]{xy}

\newtheorem{theorem}{Theorem}[section]
\newtheorem{lemma}[theorem]{Lemma}
\newtheorem{proposition}[theorem]{Proposition}
\newtheorem{corollary}[theorem]{Corollary}
\newtheorem{question}[theorem]{Question}
\theoremstyle{definition}
\newtheorem{definition}[theorem]{Definition}

\theoremstyle{remark}
\newtheorem{remark}[theorem]{Remark}
\newtheorem{claim}{Claim}

\newtheorem{case}{Case}
\newtheorem{step}{Step}

\numberwithin{equation}{section}

\makeatletter
\def\Fix{\mathop{\operator@font Fix}\nolimits}
\def\Sh{\mathop{\operator@font Sh}\nolimits}
\def\id{\mathop{\operator@font id}\nolimits}
\def\st{\mathop{\operator@font st}\nolimits}
\def\RP{\mathop{\operator@font \mathbb{R} P}\nolimits}
\def\TRP{\mathop{\operator@font T\mathbb{R} P}\nolimits}
\def\int{\mathop{\operator@font int}\nolimits}
\def\AR{\ensuremath{\mathsf{AR}}}
\def\ANR{\ensuremath{\mathsf{ANR}}}
\def\id{\ensuremath{{\mathop{\operator@font id}\nolimits}}}
\def\LC{\ensuremath{{\mathop{\operator@font LC}\nolimits}}}
\def\int{\ensuremath{{\mathop{\operator@font int}\nolimits}}}
\def\ker{\ensuremath{{\mathop{\operator@font ker}\nolimits}}}
\def\bd{\ensuremath{{\mathop{\operator@font bd}\nolimits}}}
\def\disp{\ensuremath{{\mathop{\operator@font disp}\nolimits}}}
\def\displ{\ensuremath{{\mathop{\operator@font disp}\nolimits}}}
\def\dom{\ensuremath{{\mathop{\operator@font dom}\nolimits}}}
\def\interior{\ensuremath{{\mathop{\operator@font Int}\nolimits}}}
\def\boundary{\ensuremath{{\mathop{\operator@font Bd}\nolimits}}}
\def\identity{\ensuremath{{\mathop{\operator@font id}\nolimits}}}
\def\diam{\ensuremath{{\mathop{\operator@font diam}\nolimits}}}
\def\sgn{\mathop{\operator@font sgn}\nolimits}
\makeatother

\def\mathcal#1{\mathscr{#1}}

\begin{document}

\title[Fixed-Point Sets]{Fixed-point sets of semifree group actions on infinite-dimensional compact convex sets and their manifolds}

\author{Jan van Mill}
\address{KdV Institute for Mathematics,
University of Amsterdam,
Science Park 105-107,
P.O. Box 94248,
1090 GE Amsterdam, The Netherlands}
\email{j.vanMill@uva.nl}

\author{James E. West}
\address{Department of Mathematics, Cornell University, Ithaca, NY
14853-4201, USA}
\email{jwestcornell@yahoo.com}

\date{\today}



\begin{abstract}
Let $A$ be the fixed point set of a semifree action of a compact (metric) group $G$  on a Keller space $X$ (an
infinite-dimensional compact convex subset of a separable Fr\' echet space).
All Keller spaces are homeomorphic and all nontrivial compact Lie groups act on all $X$'s semifreely with unique fixed points (see Section 3).  We prove (1) if $A$ is a single point, then it is the fixed point set of a semifree action of $G$ with uniformly small orbits, (2) if $A$ has Property Z (e.g.,
is of infinite codimension) in $X$, then the action of $G$ may be replaced by one with fixed point set $A$ and uniformly small orbits, and (3) if $A$ is an \ANR, then $G$ acts on $X$ semifreely with fixed point set homeomorphic to any compact (metric) space that is Shape equivalent to $A$.  Corollaries are (a) every Cell-like set (i.e., continuum of
trivial Shape) embedded in a Keller space $X$ as a Z-set is the fixed point set of semifree actions on $X$ with uniformly small orbits of all nontrivial compact Lie groups (Corollary \ref{CE}) and (b) (a partial
converse to a Theorem of P.\ A.\ Smith) if a compact (metric) space $Y$ has the Shape of a compact \ANR, then $Y$ embeds in $X$ as the fixed point set of a homeomorphism of $X$ of
period $n$ if it is acyclic in \v Cech homology with $\mathbb{Z}_{n}$ coefficients.  Finally, we (4) apply recent results of Cappell,
Weinberger, and Yan \cite{cwy} to give for a compact manifold $M$ modeled on a Keller space homological and algebraic K-theoretic conditions equivalent to the existence of semifree actions of
finite groups $G$ on $M$ with prescribed \ANR\ fixed point sets and (5) generalize (3) to prove that if
$A$ is an \ANR\ that is the fixed point set of a semifree action of $G$ on $M$, then every compact
(metric) space that is Shape equivalent to $A$ is homeomorphic to the fixed point set of a
semifree action of $G$ on $M$.
\end{abstract}

\maketitle
\section{Introduction}\label{introduction}

By \emph{space} we will always mean separable metrizable topological space.

 If $G$ is a topological group, $X$ is a topological space, $\alpha$ is a transformation group action of $G$
 on $X$, and there is a subset $Y$ of $X$ such that
 for each element $g$ in $G$ which is not the identity,

 $$    Y = \{x \in X : \alpha(g, x) = x\},
 $$
 then $\alpha$ is said to be a \emph{semifree action} of $G$ on $X$ with \emph{fixed point set}, $
 \Fix(\alpha)$, equal to $Y$.
 If $\Fix(\alpha)$ is empty, then $\alpha$ is said to be \emph{free}. If $\Fix(\alpha)$ is a single point,
 $\alpha$ is called \emph{based-free}.

The second author~\cite{West69} proved that if $X$ is a space with the so-called \emph{reflective isotopy property}, and
 if $G$ is a topological group which admits a free action on its infinite Cartesian power $X^\infty$, then every closed subset $A$ of
 $X^\infty$ is the fixed point set of some semifree action of $G$ on $X^\infty$. This theorem can be applied
to get interesting actions on infinite-dimensional linear spaces, but not on spaces which have, for example, the fixed point
property.
   We are interested here in semifree group actions on infinite-dimensional \emph{compact} convex
  subspaces of Fr\'echet spaces, which we call
{\em Keller spaces}, and in compact manifolds modeled on them.
By a theorem of Keller~\cite{keller} (for Hilbert spaces, extended in ~\cite[Proposition 3.1]{bp}), each
infinite-dimensional compact convex subspace of a Fr\'echet space is homeomorphic to
the Hilbert cube $Q=\prod_{n=1}^\infty [-1,1]_n$ with the standard product topology.  Hence, topologically,
we are dealing with just one space.  The product structure of $Q$ is very convenient, and
we work in that model.  As $Q$ is a product of intervals, it is immediately seen to be a compact Absolute
Retract (\AR) and to have the
fixed point property; moreover, by another theorem of Keller~\cite{keller} (see also Fort \cite{Fort62}),
it is \emph{topologically homogeneous}, i.e., for any
two points $x$ and $y$ there is a homeomorphism $f$ of $Q$ with $f(x)=y$. Thus, all Keller spaces have the
fixed point property, are \AR's, and are homogeneous.
A space that is homeomorphic to $Q$ is called {\em a Hilbert cube}.

We will now describe a few illuminating examples that will motivate the rest of this paper.
Let $\sigma_Q: Q\to Q$ be the linear involution $x\mapsto -x$. Our first observation is that if $X$ is a
compact \AR, then it is the
fixed point set of the involution $\sigma_Q\times \identity_X : Q\times X \to Q\times X$, which is a
Hilbert cube by Edwards' Theorem (see Section \ref{infdim}). By the same
argument, using the product action, if $X$ is the fixed point set of some involution on some compact \AR, then it is the fixed point set
of involutions on all Keller spaces.

An analogous argument
replacing $\sigma_Q$ by a product of rotations of order $n$ and $X$ by the
join $\mathbb Z_n\ast P$ with $\mathbb Z_n$ acting on itself shows that if $P$ is a noncontractible
finite simplicial complex that is acyclic in integral homology ({\it i.e.,} $H_i(P)=0 $ for
$i > 0$ and $H_0(P) =\mathbb Z$) then $P$ is an example of a
noncontractible Absolute Neighborhood Retract (\ANR) which is the fixed point set of semifree actions of
homeomorphisms of all periods on Keller spaces.  (The suspension $\Sigma P$ is contractible \cite[Section 4.2, Exercise 8]
{Hatcher01},
hence an \AR,
and by induction on $n$, so is $\mathbb{Z}_n\ast P = \mathbb{Z}_n\times [0,1]\times P/\sim$ where the equivalence relation $\sim$ is defined by
$(k,0,x_1)\sim (k,0,x_2)$ and $(k_1,1,x)\sim (k_2,1,x)$.)
That such a space $P$ exists is
well-known and follows from work of Poincar\'e~\cite{SeifertThrelfall}. For a nice ``minimal'' construction, see Lutz
and Ziegler~\cite{LutzZiegler}.  The diagonal $\mathbb{Z}_n$-action on $(\mathbb{Z}_n\ast P)^{\infty}$, which is a Hilbert cube by \cite[Theorem 6.2]{West70}, shows that the countably infinite product $X=P^\infty$ is a fixed point set as well.
It has rather ``bad'' topological properties: it is not locally contractible, it does not have trivial
Shape (= pro-homotopy type), etc.

This example generalizes to all compact groups $G$ that act based-freely on $Q$, in particular, to all
nontrivial compact Lie groups (Theorem \ref{standard action}), and all acyclic compact \ANR's (Theorem \ref{acyclicANR}).

The conclusion is that finding necessary and sufficient topological conditions for a
space to be a fixed point
set of semifree compact group actions on Keller spaces promises to be rather difficult.

Fortunately, Smith~\cite[Theorem III]{Smith41} proved that
if $X$ is compact and acyclic in reduced \v Cech  homology with $\mathbb Z_p$, $p$ prime, and $h$ is a homeomorphism of $X$ of period
$p$, then the set of fixed points of $h$ is $\mathbb Z_p$-acyclic as well (in particular, it is a continuum).

Jones~\cite{Jones71} gave the following as a converse to Smith's Theorem: if $K$ is a finite, connected
CW-complex and for some $n > 1 \,\,\tilde H_i(K;\mathbb{Z}_n)=0$ for all $i\ge 0$, then there is a contractible, finite CW-complex $L$
containing $K$ (as a subcomplex) and a semifree action of $ \mathbb{Z}_n$ on $L$ by cellular maps with
fixed point set $K$.  Assadi and Vogel \cite{av} considered replacing $L$ by a manifold and introduced an
algebraic K-theoretic obstruction group which has been recently exploited by Cappell,
Weinberger, and Yan \cite{cwy}  to give necessary and sufficient conditions for all finite groups $G$ that for a finite cell complex $Y$ with a semifree $G$ action and fixed point set $F$, a map $f:A\to F$
where $A$ is a finite cell complex extends to an invariant homotopy equivalence (or simple homotopy
equivalence) $\bar f:X\to Y$ where $X$ is a finite cell complex with a semifree cellular
$G$ action with fixed point set $A$. They also give a generalization to actions on compact \ANR's.

 The first aim of this paper is to show that based-free actions of compact groups on Keller spaces can be
 significantly ``improved''.

 The second aim of this paper is to apply these improved actions in the setting of more general fixed point sets
 in Keller spaces.
Even for fixed point sets of periodic
 homeomorphisms the situation is complicated.  For example, R.\ D.\ Anderson's conjecture that all based-free
 involutions of Keller spaces are conjugate is still open after 60 years.  (See~\cite{Wong74, BersteinWest78,
 WestWong79, Antonyan99, millwest19} for more information about the validity of
the conjecture.)  Floyd's examples  in \cite{Floyd56} of a period 6 simplicial
homeomorphism of a contractible finite simplicial complex whose fixed point set has two components,
a 2-sphere and an arc, and another period 6 homeomorphism on a homology 2-sphere with a fixed point
set that is not locally
connected show that the passage from semifree actions to more general ones is more complicated than might be expected.  By the method outlined above, we conclude that every Keller space has period
6 homeomorphisms with disconnected fixed point sets and with non locally connected fixed point sets.  Therefore we restrict attention to the semifree case.

In this part of the paper we show that if $X$ is Cell-like, every nontrivial compact Lie group acts semifreely on each Keller space with  fixed point set homeomorphic to $X$ and build on Jones' result to give a converse in Keller spaces to P.A. Smith's Theorem for spaces that are Shape equivalent to \ANR's.

Our third aim is to study semifree actions on compact spaces $M$ modeled on (i.e., locally homeomorphic to) Keller spaces.  These are
called Hilbert cube manifolds ($Q$-manifolds) in the literature and we shift to that notation here.
In Section~6, we show how the K-theoretic invariants identified by Cappell, Weinberger, and Yan
in \cite{cwy} apply to $Q$-manifolds.  This allows us to apply their results about finite
group actions on cell complexes and \ANR's to generalize many of our results on Keller spaces
to compact $Q$-manifolds.

\subsection{Some Infinite-Dimensional topology of Keller spaces}\label{infdim}
The group of all self-homeomorphisms
of~$Q$ with the compact-open topology is homeomorphic to separable (infinite-dimensional) Hilbert
space  and the homeomorphism groups of compact $Q$-manifolds are Hilbert manifolds  \cite{Ferry77, Torunczyk77}, so this is true of all Keller spaces as well.
There are many interesting spaces
besides Keller spaces that are Hilbert cubes.  Examples include (a) $X\times Q$ for $X$ any compact,
contractible simplicial complex \cite{West70} or dendron (R. D. Anderson (unpublished), see~\cite{West70}), or, indeed,
any compact \AR\ (Edwards' Theorem (unpublished), which states that if $X$ is an \ANR, then $X\times Q$ is a $Q$-manifold, see  \cite[Theorem 44.1]{chapman:lectures}, \cite[Theorem 7.8.1]{vm:book}, \cite{Tor:hilbertcube}),
and (b) the space of closed nonvoid subsets of any non-degenerate Peano continuum \cite{schori:west2},
\cite{cusc74}.  Since a countably infinite product of nondegenerate Hilbert cube factors is a Hilbert cube \cite{West70}, any countably infinite product of nondegenerate compact \AR's is a Hilbert cube.

Many locally compact spaces of functions are locally homeomorphic to the Hilbert cube.
Provocative examples of spaces of maps between Riemannian manifolds that are $Q$-manifolds have been given by Geoghegan \cite{geo} and Colvin \cite{col}.  For example, if $\alpha > 0$, the $C^0$ closure of the space of $C^1$ loops with derivative norm bounded by $\alpha$ in a  manifold that is flat or of constant negative sectional curvature is a Hilbert cube manifold.  An elegant topological characterization
of Hilbert cube manifolds has been given by Toru\'nczyk
(\cite{Tor:hilbertcube}, \cite[7.8.3]{vm:book}, \cite{LevinDijkstravanMill2016}); they are the locally compact \ANR's such that any two maps of $Q$ into them can be approximated by maps with disjoint images.

Of crucial importance in Infinite-Dimensional topology is the so-called \emph{(Estimated) Homeomorphism Extension Theorem.} The first version of it is due to Anderson~\cite{topinfdef}. It was strengthened by various authors to Theorem~19.4 in Chapman~\cite{chapman:lectures}: \emph{Let $M$ be a $Q$-manifold, $A$ be locally compact and let $f: A\times [0,1]\to M$ be a proper map such that $F_0$ and $F_1$ are Z-embeddings. Then there exists an isotopy $H: M\times [0,1]\to M$ such that $H_0=\identity$ and $H_1\circ F_0=F_1$. Moreover if $F$ is limited by an open cover $\mathcal{U}$ of $M$, then we may choose $H$ to also be limited by $\mathcal{U}$.} This fundamental result will be used several times by us in the remaining part of this paper.
(A homotopy $H:M\times [0,1]\to M$ is {\em limited by $\mathcal U$} provided that for each $x\in M$ there is a $U\in \mathcal U$ containing $H(\{x\}\times [0,1]$.)

\subsection{More on the organization of the paper}
In Section 3 we show in Theorem~\ref{derdestelling} that if a nontrivial compact group
 $G$ acts based-freely on the Hilbert cube then there is a based-free action of $G$ on $Q$ arbitrarily close to the trivial action, and in
 Theorem~\ref{TheoremOnIsotopies} that sufficiently close such $G$-actions are isotopic via ``small''
 isotopies through based-free actions.

In Theorem~\ref{derdestellingPartTwo} we show that every semifree action of a compact group
 $G$ on a Keller space whose fixed point set is a Z-set can be replaced by one having the same fixed point set and
 which is arbitrarily close to the trivial action.  (See Section \ref{preliminaries} for the definition of
 Z-set.)

In Section 5 we show in Corollary \ref{CE} that every nontrivial compact Lie group acts semifreely on every
Keller space $K$ with fixed point sets homeomorphic to all continua of trivial Shape, i.e., Cell-like sets.  In Corollary
\ref{shapeanr:p-acyclic} we show
that if $X$ has the Shape of a compact \ANR\ (e.g., is $1$-UV and finite dimensional
with finitely generated integral \v Cech cohomology \cite{geolach}) then $X$ is the fixed point set of a periodic homeomorphism of $K$ with period
$n > 1$ if and only if it is acyclic in reduced \v Cech homology with $\mathbb Z_n$
coefficients.  (The ``only if'' direction for $n$ prime is due to Smith as noted above.)  These actions may be taken arbitrarily close to the identity.

Theorem \ref{QM:simple} provides necessary and sufficient conditions that a $Q$-submanifold
Z-set $A$ of a compact, connected $Q$-manifold $M$ be homeomorphic to the fixed point set of a semifree action of a finite group induced from a cellular action on a cell complex.

We conclude in Section 7 by combining the results of Sections 5 and 6.
Theorems \ref{shapeobstruction} and \ref{semitame obstruction} show that for a finite group $G$ and compact, connected $Q$-manifold $M$ with fundamental group $\pi$, if $X$ is a Z-set in $M$ with the Shape of a compact \ANR\ $A$
and the inclusion $f$ of $X$ into
$M$ induces isomorphisms of the suitably defined \v Cech homology groups  with (local) coefficients in $\mathbb{F}_p\pi$ for all primes
$p$ dividing the order of $G$, then there is a well-defined obstruction in $\tilde K_0^{top}
(\mathbb{Z} [\pi] \subset \mathbb{Z} [\pi \times G])$ the vanishing of
which implies that there is an embedding of $A$ in $M$ satisfying

\begin{enumerate}

\item there is a compact $Q$-manifold $N$ containing $A$ as a Z-set,

\item there is a semifree action $\alpha$ of $G$ on $N$ with fixed point set $A$,

\item  there is an invariant homotopy equivalence $g:N\to M$ that is the identity on $A$, and

\item if the Whitehead torsion of $g$ is in the subgroup $f_{\ast}(Wh(\check \pi_1(X,x_0)))$ of $Wh(\pi)$, where $\check \pi_1(X,x_0)$ is the inverse limit of fundamental groups of neighborhoods of $X$, then $X$ is the fixed point set of a semifree action of $G$ on $M$.
\end{enumerate}

\newpage
\tableofcontents

\newpage
\section{Preliminaries}\label{preliminaries}

All topological groups we deal with are compact. For a set $X$, $\identity_X: X\to X$ denotes the identity
function, ``$\approx$" means either ``is a homeomorphism" or ``is homeomorphic to", and ``$\simeq$" means
either ``is a homotopy equivalence" or ``is homotopic to".  We use ``deformation retraction" to mean a retraction,
$r$ that is homotopic to the identity by a homotopy that is stationary on the image of $r$,
 $\ANR$ and $
\AR$ to denote Absolute Neighborhood Retract and Absolute Retract in the Category of all
(compact metric) spaces,
and $G$-\ANR\ and $G$-\AR\ to denote Equivariant Absolute Neighborhood Retract and Equivariant Absolute
Retract in the Category of all spaces with actions of a particular group $G$ and equivariant maps.

For
a discussion of the particular complexities of these latter two notions see \cite{Antonyan17}, and
\cite{Antonyan22}.  A compact space is {\em CE} (Cell-like or Cellullar Equivalent) if it deforms to a point in each of its neighborhoods when it is embedded in an \ANR.  This is easily seen to coincide with the property of having the Shape of a point (see Section 7).  A map is CE or Cell-like provided that it is surjective and each point inverse is Cell-like (see \cite{lacher}).

The monographs \cite{chapman:lectures}, \cite{vm:book}, see also \cite{Sakai13} and \cite{Sakai20}, which includes some
infinite-dimensional results from \cite{chapman:lectures} and \cite{vm:book}, each form a good background for the infinite dimensional topology used in this paper.

A fundamental concept to our discussion is Anderson's notion of a Z-set. In our context where $X\in \ANR$
is compact, the simplest definition is that a closed set $A$ in $X$
is a \emph{Z-set in $X$} provided that every continuous function from $Q$ to $X$ can be
approximated arbitrarily closely by continuous functions from $Q$ to $X\setminus A$. That is, for every
continuous function $f: Q\to X$ and $\varepsilon > 0$ there is a continuous function $g :  Q\to X\setminus
A$ such that for every $q\in Q$, the points $f(q)$ and $g(q)$ are $\varepsilon$-close.
This concept arose as a topological characterization of closed subsets of Keller spaces $X$ that can be moved
into infinite codimensional subspaces by homeomorphisms of $X$ \cite{topinfdef}.

Set $\mathbb{I} = [0,1]$, $\mathbb{J} = [-1,1]$ and
$Q= \prod_{i=1}^\infty \mathbb{J}_i$ with the standard product topology. The metric $d(x,y) =
\sum_{i=1}^\infty 2^{-i-1}|x_i-y_i|$ generates its topology.  On $Q^\infty$, the countably infinite product of copies of $Q$, we use the metric
$\varrho(x,y) = \sum_{n=1}^\infty 2^{-n} d(x_n,y_n)$.

For $n \ge 0$, denote $\prod_{i> n}\mathbb{J}_i$ by $Q_n$, hence $Q$ is $Q_0$. A point $x\in Q_n$ is
represented by its coordinates $(x_{n+1},x_{n+2}, \dots)$. On $Q_n$ we use the standard metric $d(x,y) =
\sum_{i=n+1}^\infty 2^{-i-1} |x_i-y_i|$.
For $n \ge 1$, we can factor $Q$ as $\mathbb{J}^n \times Q_n$. For $x,y\in Q$ and $n \ge 0$, $x\equiv_n y$ abbreviates that for all $1\le i \le n$, $x_i = y_i$. Hence if
$n = 0$, there is no condition and so all points in $Q$ are $\equiv_0$ ``equivalent''.

If $X$ is a product of spaces each containing 0, then $\overline 0$ denotes the point all of whose
coordinates are 0.

If $(X,d)$ is a space, and $f :  X\to X$, then its \emph{displacement}, $\disp(f)$, is defined as
$\sup_{x\in X} d(x,f(x))$. Hence $\disp(f)$ is the distance from $f$ to the identity function on $X$.

For an action $\alpha :  G\times X\to X, \alpha^g(x) = \alpha(g,x)$.
 The \emph{displacement}, $\disp(\alpha) = \max \{ \disp(\alpha^g): g\in G\}$ is the distance between $
 \alpha$ and the trivial
 action. If $A = \Fix(\alpha)$, and $f :  X\to X$ is a homeomorphism,
 let $\check f = \identity_G\times f$, and set $\beta = f\circ \alpha \circ \check f^{-1}$. Then $\beta :  G\times X\to X$ is an action, and $\Fix(\beta) = f(A)$.

Let $X$ be a compact space. The \emph{cone $C(X)$ over $X$} is the space $X\times \mathbb{I}/X\times \{1\}$
endowed with the quotient topology. The vertex of the cone will always be denoted by $\star$.  In the remaining part of this paper, we will need a few times that the cone over a compact \ANR\ is an \AR, \cite[5.4.2]{vm:book}.

Since $Q$ is a Keller space, so is $C(Q)$, which is thus homeomorphic to $Q$ (cf. \cite[12.2]
{chapman:lectures}). Observe that $C(X)$ is the one-point compactification of $X\times [0,1)$. For us, it will
be geometrically convenient to think of $C(X)$ as $X\times [0,1]$, endowed with the following metric
$$
    \varrho((x,s),(y,t)) = (1-\max\{s,t\}) d(x,y) + |t{-}s|,
$$
where $d$ is any metric compatible with the topology on $X$ such that $d \le 2$, \cite{metriccone}. Observe
that any two points of $X\times \{1\}$ have $\varrho$-distance 0, and that on $X\times \{0\}$, $d$ and
$\varrho$ agree. If we think of $C(X)$ in that way,
we use the notation $C(X) = X\times _c [0,1]$.

\section{Based-free actions on the Hilbert cube}

\subsection{Existence of Based-Free Actions}\label{existence}

The next result was observed with an easy proof, repeated here for the reader, in \cite{induced}.  The question of Antonyan \cite{Antonyan99} whether a compact group acting based-freely on a Keller space must be a Lie group has remained open for over a quarter century.  We conclude this subsection with an argument that allowing semifree actions on compact $Q$-manifolds does not enlarge the set of groups under consideration.

\begin{theorem}\label{standard action}  Let $G$ be a nontrivial compact Lie group.  Then there is a based-free $G$-\AR\ action of $G$ on each Keller space.
\end{theorem}

\begin{proof}  Since $G$ is a smooth manifold, it is homeomorphic to a finite
simplicial complex \cite{whitehead}.  The cone $C(G)$ is then homeomorphic to a finite contractible simplicial
complex.  Set $Q(G) = C(G)^{\infty}$.  By \cite{West70} $C(G)$ is a Hilbert cube factor and $Q(G)$ is a Hilbert
cube.  Let $G$ act by left translation on $C(G)$ and
$\alpha$ be the diagonal action on $Q(G)$.  It is immediate from the cone structure that $Q(G)$ is a $G$-\ANR.  \end{proof}

We term this the {\em standard based-free action} of $G$ on a Hilbert cube.

Theorem \ref{point} shows that if a compact group acts semifreely and nontrivially on a compact $Q$
manifold, then even if the fixed point set is not an \ANR, $G$ acts based-freely on $Q$.

\begin{theorem}\label{point} Let $G$ be a nontrivial compact group.
If there is a nontrivial semifree $G$-action $\alpha$ of $G$
on a compact
Q-manifold then $G$
acts based-freely on all Keller spaces with arbitrary fixed points; if $\alpha$ is a $G$-\ANR\ action, then $G$
acts based-freely on all Keller spaces with $G$-\AR\ actions.
\end{theorem}

\begin{proof}
Let $\alpha$ denote a semifree action of $G$ on a compact $Q$-manifold $M$ with
fixed point set $A \neq M$.  If $A =\emptyset$ then the induced action on the cone $C(M)$ of $M$ is a based-free action on an \AR\ so as observed below the diagonal action on the Hilbert cube $C(M)^{\infty}$ is based-free.
If $A \neq \emptyset$ then the product action $\beta = \alpha \times {\identity_\mathbb{I}}
$ of $G$ on $M\times \mathbb{I}$ has fixed point
set $A \times \mathbb{I}$.  Therefore it induces a based-free action $\gamma$ on the quotient
space
$X = M \times \mathbb{I} /((A\times \mathbb{I}) \cup (M \times \{1\}))$.
 To see that $X$ is an \AR\ we first collapse $M\times \{1\}$ to a single point $v$, thus obtaining
 the cone $C(M)$. This cone is an \AR\ by \cite[5.4.2]{vm:book}, \cite[6.5.7]{Sakai13}.  Let
 $k:C(M)\to X$ collapse $C(A)$ to the image $\ast$ of
 $(A\times \mathbb{I}) \cup (M \times \{1\})$.  For a map $\phi:M\to \mathbb{I}$ such that
 $\phi^{-1}(1) = A$, the function $k(x,t) = (x,s)$ where $s = \min\{t,\phi(x)\}$ embeds $X$ in $C(M)$ as a retract so it is an \AR.

By Edwards' Theorem (see Section \ref{infdim}), $X\times Q \approx Q$, so by \cite[Theorem 6.2]{West70}, $X^{\infty}\approx Q$.

The diagonal action of $\gamma$ is based-free.  Since all Keller spaces are homeomorphic to
$Q$ and homogeneous, the result follows.
If $\alpha$ is a G-\ANR\ action, so is $\beta$ and then $\gamma$ is a $G$-\AR\ action (by the
same argument given above for $X$ and the trivial action) and so is its diagonal action.
\end{proof}

\subsection{Small actions}\label{smallactions}

We now show that the actions in Theorem~\ref{point} can be ``improved'' considerably.

\begin{theorem}[small actions]\label{derdestelling}
If the compact group $G$ acts based-freely on $Q$, there is a based-free action of $G$ on $Q$ arbitrarily close
to the trivial action.
\end{theorem}

\begin{proof}
Our actions will be defined on $C(Q)$.

To begin with, factor $Q$ as $\mathbb{J}^n\times Q_n$. Fix a based-free action $\sigma_n$ of $G$ on
$Q_n$.  By homogeneity of $Q_n$, we may assume that $\Fix(\sigma_n) = \{\overline 0\}$.
Define an action $\alpha_n :  G\times Q\to Q$ by $\alpha_n(g,(x,y)) = (x,\sigma_n(g,y))$.
Then $\Fix(\alpha_n)$ is the Z-set  $J^n\times \{\overline 0\}$, and $\disp(\alpha_n)\le 2^{-n}$. By a shrinking argument with attention to orbit sizes,
we will shrink for sufficiently large $n$, $\Fix(\alpha_n)$ to a point. Our use of cones simplifies the technical
details of the construction.

So now let Let $0 < \varepsilon \le 1$, and fix $n_0$ with $2^{-n_0} < \frac{\varepsilon}{4}$. Consider the cone $C(Q)=Q\times_c [0,1]$. Define $\hat
\alpha_{n_0} :  G\times C(Q)\to C(Q)$ by $\hat\alpha_{n_0}(g,(x,t)) = (\alpha_{n_0}(g,x),t)$, and put
$F=\Fix(\alpha_{n_0})$. Then
$\disp(\hat\alpha_{n_0})=\displ(\alpha_{n_0}) <
\frac{\varepsilon}{4}$, and $\Fix(\hat\alpha_{n_0}) = F \times_c [0,1]$. Clearly, $\Fix(\hat\alpha_{n_0})$ is a
Z-set in $C(Q)$. Indeed, $F \times [0,1{-}\frac{1}{n}]$ is a Z-set in $C(Q)$ for every $n$, \cite[7.2.5]
{vm:book}, and so is the vertex of the cone, hence $\Fix(\hat\alpha_{n_0})$ is a  closed countable union of
Z-sets, thus a Z-set.

Consider the set $\hat F = F \times_c [\delta,1]$, where $0<\delta<\frac{1}{2}$ is so small that the
homeomorphism $F \times_c [0,1]\to F \times_c [\delta,1]=\hat F$ defined by $(x,t) \mapsto (x,(1{-}\delta) t{+}
\delta)$ can be extended to a homeomorphism $f$ of $C(Q)$ such that $\disp(f) < \frac{\varepsilon}{8}$,
\cite[11.1]{chapman:lectures}. As in Section \ref{preliminaries}, let $\check f = \identity_G\times f$, and set $\beta = f\circ \hat\alpha_{n_0} \circ \check f^{-1}$.
Then, $\Fix(\beta) = \hat F$, and
$$
    \disp(\beta) \le 2\disp(f) + \disp(\hat\alpha_{n_0}) < 2 \textstyle\frac{\varepsilon}{8} + \textstyle\frac{\varepsilon}{4}=\textstyle\frac{\varepsilon}{2}.
$$
Using the Venetian Blind Lemma~\ref{JimLemma03}, one may shrink $\hat F$ to the vertex of $C(Q)$ while maintaining control over the sizes of the orbits, i.e., construct a map $\zeta :  C(Q)\to C(Q)$ such that
\begin{enumerate}
\item $\zeta^{-1}(\star) = \hat F$,
\item $\hat F$ is the only non-degenerate point inverse of $\zeta$, and
\item for each $x\in C(Q)$ and $g\in G$, $\varrho(\zeta(x),\zeta(\beta(g,x)) < \frac{\varepsilon}{2} + \disp(\beta) = \frac{\varepsilon}{2} + \frac{\varepsilon}{2} = \varepsilon$.
\end{enumerate}
Then with $\check \zeta=\identity_G\times \zeta$, $\gamma = \zeta \circ \beta \circ \check \zeta^{-1}$ is a
based-free $G$-action on $C(Q)$ with unique fixed point $\star$ and displacement less than $\varepsilon$.
\end{proof}

We now proceed to the proof of the Venetian Blind Lemma. Our notation is as in the proof of Theorem~\ref{derdestelling}.

Let $\mu > 0$.

Since $\star$ is a fixed point of $\beta$, there exists $\frac{1}{2} < \nu< 1$ such that
$\diam_\varrho(Q\times_c [\nu,1]) < \frac{\mu}{8}$, while moreover, for all $x \in Q\times_c [\nu,1]$ and $g\in G$,
$\varrho(x,\beta(g,x)) < \frac{\mu}{8}$.

As warming up exercise, we first prove the following.

\begin{lemma}\label{preJimLemma03}
For every $0 < \nu < d < 1$, there is a map $h :  C(Q)\to C(Q)$ such that
\begin{enumerate}
\item $h^{-1}(\star) = F\times_c [d,1]$,
\item $F\times_c [d,1]$ is the only non-degenerate point inverse of $h$,
\item $h$ is supported on $Q\times_c [\nu,1]$, and
\item for each $x\in Q\times_c [\nu,1]$ and $g\in G$, $\varrho(h(x),h(\beta(g,x)) < \frac{\mu}{2}$.\label{devierdepre}
\end{enumerate}
\end{lemma}

\begin{proof}
Let $f :  [0,1] \to [0,1]$ be the piecewise linear map that is the identity on $[0,\nu]$, sends $[\nu,d]$ onto $[\nu,1]$
and collapses $[d,1]$ to 1.
Let $H :  [0,1]\times [0,1]\to [0,1]$ be defined by $H_s(t) = (1{-}s) t + s f(t)$. Then $H_0$ is the identity,
$H_1 = f$, and each $H_s$ for $s < 1$ is a homeomorphism, being the sum of two increasing functions, one
of them strictly increasing. Let $\phi :  Q\to [0,1]$ be continuous with $\phi^{-1}(1) = F$. Now
define $h$ by $h(x,t) =
(x,H_{\phi(x)}(t))$. We clearly only need to check (\ref{devierdepre}). So, let $x\in Q\times_c [\nu,1]$. Since
$h(x) \in Q\times_c [\nu,1]$,
$\varrho(x,h(x))< \frac{\mu}{8}$.
If $\beta(g,x)\not\in Q\times_c[\nu,1]$, then $\varrho(\beta(g,x),h(\beta(g,x))=0$. And if $\beta(g,x)\in
Q\times_c[\nu,1]$ then, by what we just observed, $\varrho(\beta(g,x),h(\beta(g,x)) < \frac{\mu}{8}$. In
either case,
$$
    \varrho(h(x), h(\beta(g,x)) \le \varrho(h(x),x) + \varrho(x,\beta(g,x)) + \varrho(\beta(g,x), h(\beta(g,x)))
    <\textstyle\frac{3\mu}{8}.\qedhere
$$
\end{proof}

\begin{lemma}[Venetian Blind]\label{JimLemma03}
There is a map $\zeta :  C(Q)\to C(Q)$ such that
\begin{enumerate}
\item $\zeta^{-1}(\star) = \hat F$,
\item $\hat F$ is the only non-degenerate point inverse of $\zeta$, and
\item for each $x\in C(Q)$ and $g\in G$, $\varrho(\zeta(x),\zeta(\beta(g,x))) < \mu + \disp(\beta)$.
\end{enumerate}
\end{lemma}

\begin{proof}
Fix a positive integer $m$ for which there exists a positive integer $k$ such that
\begin{enumerate}
\item $\frac{k}{m} < \delta < \frac{k+1}{m}<\nu$, $\frac{2}{m} < 1{-}\nu$, and $\frac{6}{m} < \frac{\mu}{4}$.
\end{enumerate}
\noindent Let
$j_1 = \max\{j \in\mathbb{N} : \frac{k+j}{m} \le \nu\}$. Hence $\frac{k+j_1}{m} \le \nu < \frac{k+j_1+1}{m} <
\frac{k+j_1+2}{m} < 1$.

By induction on $1\le i\le  j_1{+}1$, we will construct homeomorphisms $h_1,\dots,h_{j_1+1}$ of $C(Q)$, real numbers
$c(1),c(2),\dots, c(j_1{+}1)$ with
$$
    \textstyle\frac{k}{m} < c(1) < \delta < c(2) < \textstyle\frac{k+1}{m} < c(3) < \textstyle\frac{k+2}{m} < \cdots <
    \textstyle\frac{k+j_1-1}{m} < c(j_1{+}1) < \textstyle\frac{k+j_1}{m},
$$
closed neighborhoods
$$
    B(1)\supseteq B(2) \supseteq \cdots\supseteq B(j_1{+}1)
$$
of $F$ in $Q$, and open subsets
$$
    U(1),U(2), \dots, U(j_1{+}1)
$$
of $C(Q)$
such that for every $i \le j_1{+}1$, if $\gamma_i = h_i\circ\cdots\circ h_1$, then
\begin{enumerate}[resume*]
\item $h_i$ is supported on $K(i) = B(i) \times [c(i),\frac{k+i+1}{m}] \subseteq U(i)$,\label{supported}
\item $h_1(F\times [\delta,\frac{k+2}{m}]) = F\times [\frac{k+1}{m},\frac{k+2}{m}])$, and for $i > 1$,\\ $h_{i}(F\times [\frac{k+i-1}{m},\frac{k+i+1}{m}]) = F\times [\frac{k+i}{m},\frac{k+i+1}{m}],$
\item $\varrho(h_i,\identity_{C(Q)})\le  \frac{2}{m}$, \label{afstand}
\item $U(1) =\{x\in C(Q): \diam_{\varrho}(\beta(G\times \{x\}))< \mu\}$, and for $i > 1$, $U(i)$ is the set of all $x\in
C(Q)$ such that \label{deUs}
      $
      \diam_{\varrho} (\gamma_{i-1}\circ\beta\circ\check\gamma_{i-1}^{-1}(G\times\{x\}))< \textstyle\frac{1}{2}\mu
      $
      (we again use the notation introduced in Section \ref{preliminaries}),
\item $\gamma_i(\hat F) = F\times_c [\frac{k+i}{m},1]$,
\item for each $x\in C(Q)$ and $g\in G$, $\varrho(\gamma_i(x),\gamma_i(\beta(g,x))) \le \frac{6}{m} + \frac{1}
{2}\mu + \varrho(x,\beta(g,x))$.\label{laatste}
\end{enumerate}

\begin{step}
$i=1$.
\end{step}

For every $x\in \hat F$ and $g\in G$, $\varrho(x,\beta(g,x))=0$, hence
$U(1)$ defined as in (\ref{deUs})
is an open neighborhood of $\hat F$ in $C(Q)$. Hence it is clear that we may pick $\frac{k}{m} < c(1)<\delta $
and $B(1)$ such that $K(1) = B(1) \times [c(1),\frac{k+2}{m}] \subseteq U(1)$.
By the same technique as in the proof of Lemma~\ref{preJimLemma03} (use a Urysohn function that is 0 off
the interior of $B(1)$ and precisely $1$ on $F$), there is a homeomorphism $h_1 :  C(Q)\to C(Q)$ such
that
\begin{enumerate}[resume*]
\item $h_1$ is supported on $K(1)$,
\item $h_1(F\times [\delta,\frac{k+2}{m}] = F\times [\frac{k+1}{m},\frac{k+2}{m}]$, \label{jawelledrie}
\item if $(x,t) \in K(1)$, then for some $t\le t'\le \frac{k+2}{m}$, $h_1(x,t)=(x,t')$.\label{jawellevier}
\end{enumerate}
Observe that by (\ref{jawelledrie}),
$h_1(\hat F) = F\times_c [\frac{k+1}{m},1]$, and so, $h_1(\hat F) \subseteq \hat F$,\label{waarisFstapeen} and by
(\ref{jawellevier}),
$\varrho(h_1,\identity_{C(Q)})\le \frac{k+2}{m} - c(1)\le \frac{k+2}{m} - \frac{k}{m} = \frac{2}{m}$.

We claim that
for each $x\in C(Q)$ and $g\in G$, $\varrho(h_1(x),h_1(\beta(g,x))) \le \frac{4}{m} + \varrho(x,\beta(g,x))$. \label{jawellezeven}
To prove this, we basically repeat the same arguments as in the proof of Lemma~\ref{preJimLemma03}. Let
$x\in C(Q)$ and $g\in G$. If both $x$ and $\beta(g,x)$ are in $C(Q)\setminus K(1)$, then $h_1(x) = x$ and $h_1(\beta(g,x)) =
\beta(g,x)$, hence $\varrho(h_1(x),h_1(\beta(g,x))) = \varrho(x,\beta(g,x))$. If $\beta(g,x)\in C(Q)\setminus K(1)$ and
$x \in K(1)$, then $h_1(\beta(g,x))=\beta(g,x)$ and $\varrho(x,h_1(x)) \le \frac{2}{m}$, thus
$$
    d(h_1(x),h_1(\beta(g,x))) \le \varrho(h_1(x),x) + \varrho(x, \beta(g,x))\le \textstyle\frac{2}{m} + \varrho(x,
    \beta(g,x)).
$$
Since $\beta$ is an action, this also deals with the case that $\beta(g,x) \in C(Q)\setminus K(1)$ and $x\in
K(1)$. It remains to check what happens if $x,\beta(g,x)\in K(1)$. Now, $\varrho(x,h_1(x)) \le \frac{2}{m}$ and $
\varrho(\beta(g,x), h_1(\beta(g,x)) \le \frac{2}{m}$. Hence $\varrho(h_1(x),h_1(\beta(g,x)) \le \frac{4}{m}+ \varrho(x,\beta(g,x))$.

We conclude that our choices satisfy (\ref{supported}) through (\ref{laatste}) for $i=1$.

\begin{step}
From $i$ to $i{+}1$.
\end{step}

Assume that we completed the construction for some $i \le j_1$.
If $y\in \gamma_i(\hat F)$, then $\gamma_i^{-1}(y)\in \hat F$, and so $\gamma_i(\beta(g,\gamma_i^{-1}(y))) = y$.
Hence $U(i{+}1)$ defined as in $
(\ref{deUs})_{i+1}$ is an open neighborhood of $\gamma_i(\hat F) = F\times_c [\frac{k+i}{m},1]$. Hence it is
clear that we may pick $\frac{k+i-1}{m} < c(i{+}1)< \frac{k+i}{m}$ and $B(i{+}1)$ such that $K(i{+}1) = B(i{+}1)
\times [c(i{+}1),\frac{k+i+2}{m}] \subseteq U(i{+}1)$.

There is a homeomorphism $h_{i+1} :  C(Q)\to C(Q)$ constructed in the same way as $h_1$, such that
\begin{enumerate}[resume*]
\item $h_{i+1}$ is supported on $K(i{+}1)$,
\item $h_{i+1}(F\times [\frac{k+i}{m},\frac{k+i+2}{m}]) = F\times [\frac{k+i+1}{m},\frac{k+i+2}{m}]$,
\item $\varrho(h_{i+1},\identity_{C(Q)})\le  \frac{2}{m}$.\label{jawelledertien}
\end{enumerate}
Hence $h_{i+1}(\gamma_i(\hat F)) = F\times_c [\frac{k+i+1}{m},1]$.

 We claim that
for each $x\in C(Q)$ and $g\in G$,
$$
    \varrho(h_{i+1}(\gamma_i(x)),h_{i+1}(\gamma_i(\beta(g,x)))) \textstyle\le \frac{6}{m} + \textstyle\frac{1}{2}\mu + \varrho(x,\beta(g,x)).
$$
We distinguish between several subcases.

\begin{case}
Both $x$ and $\beta(g,x)$ are not in $\bigcup_{j\le i+1} K(j)$.
\end{case}

Then $h_{i+1}(\gamma_i(x)) = x$ and $h_{i+1}(\gamma_i(\beta(g,x))) = \beta(g,x)$, hence there is nothing to prove.

\begin{case}
$\beta(g,x)\not\in \bigcup_{j\le i+1} K(j)$ but $x\in \bigcup_{j\le i+1} K(j)$.
\end{case}
Clearly, $h_{i+1}(\gamma_i(\beta(g,x))) = \beta(g,x)$. If $x\not\in K_{i+1}$, then since $h_{i+1}$ is supported
on $K(i{+}1)$, by $(\ref{laatste})_i$ there is nothing to prove. So assume that $x\in K(i{+}1)$. Observe that $
\bigcup_{j\le i-2}K(j) \cap K(i{+}1) = \emptyset$. Hence $h_{i+1}(\gamma_i(x))) = h_{i+1}(h_i(h_{i-1}(x)))$,
hence by $(\ref{afstand})_{i-1,i}$ and
(\ref{jawelledertien}), $\varrho(h_{i+1}(\gamma_i(x)),x)\le \frac{6}{m}$. Hence this is again as required.

Since $\beta$ is an action, this also deals with the case $x\not\in \bigcup_{j\le i+1} K(j)$ but $\beta(g,x)\in
\bigcup_{j\le i+1} K(j)$.

\begin{case}
Both $x$ and $\beta(g,x)$ are in $\bigcup_{j\le i+1} K(j)$.
\end{case}

Assume first that $x\in \bigcup_{j\le i} K(j)$. Assume first that $y=\gamma_i(x) \in K(i{+}1) \subseteq U(i{+}1)$. Let $e_G$ denote the neutral element of $G$ and let $g\in G$ be arbitrary. By $(\ref{deUs})_{i+1}$,
$
      \varrho (\gamma_{i}\circ\beta\circ\check\gamma_{i}^{-1}(e_G,y),\gamma_{i}\circ\beta\circ\check\gamma_{i}^{-1}(g,y))< \textstyle\frac{1}{2}\mu.
$
But $\gamma_{i}\circ\beta\circ\check\gamma_{i}^{-1}(e_G,y) = \gamma_{i}\circ\beta (e_G,x)=\gamma_i(x)$ and
$\gamma_{i}\circ\beta\circ\check\gamma_{i}^{-1}(g,y)=\gamma_i(\beta(g,x))$. Hence $\varrho(\gamma_i(x),\gamma_i(\beta(g,x))) < \frac{1}{2}\mu$ and so by (\ref{jawelledertien}),
$$
    \varrho(h_{i+1}(\gamma_i(x)),h_{i+1}(\gamma_i(\beta(g,x))) < \textstyle\frac{4}{m} + \frac{1}{2}\mu.
$$
If $\gamma_i(\beta(g,x))\in K_{i+1}$, then the same argument applies with $y = \gamma_i(\beta(g,x))$ since $
\beta$ is an action.
Hence assume that $\gamma_i(x), \gamma_i(\beta(g,x))\not\in K_{i+1}$. But then these points are not moved by
$h_{i+1}$ and so we are done by $(\ref{laatste})_i$.

Since $\beta$ is an action, this also deals with the case $\beta(g,x)\in \bigcup_{j\le i} K(j)$.

Hence we are left with the case $x,\beta(g,x)\in K_{i+1}\setminus\bigcup_{j\le i} K_i$. But then both $x$ and
$\beta(g,x)$ are not moved by $\gamma_i$, and hence we are done by $(\ref{jawelledertien})$.
Hence our choices satisfy (\ref{supported}) through (\ref{laatste}) for $i{+}1$, which completes the induction.

Now, $\gamma_{j_1+1}(\hat F) \subseteq Q\times_c [\nu,1]$. Let $h$ be the map of Lemma~\ref{preJimLemma03},
and put $\zeta = h\circ \gamma_{j_1+1}$. Then $\zeta$ satisfies (1) and (2) in the formulation of the
Venetian Blind Lemma, so we only need to worry about (3). So take an arbitrary $x\in C(Q)$ and $g\in G$.
Then by $(\ref{laatste})_{j_1+1}$,
$\varrho(\gamma_{j_1+1}(x),\gamma_{j_1+1}(\beta(g,x))) \le \frac{6}{m} + \frac{1}{2}\mu + \varrho(x,\beta(g,x))$.
Now since $\diam_\varrho(Q\times_c [\nu,1]) < \frac{\mu}{8}$ and $h$ is the identity off $Q\times_c [\nu,1]$,
if $\gamma_{j_1+1}(\beta(g,x))\not\in Q\times_c [\nu,1]$, $\varrho(\zeta(x), \zeta(\beta(g,x)) < \frac{6}{m} +
\frac{\mu}{2} + \varrho(x,\beta(g,x)) + \frac{\mu}{8}$, and if this is not the case, $\varrho(\zeta(x),
\zeta(\beta(g,x)) < \frac{6}{m} + \frac{\mu}{2} + \varrho(x,\beta(g,x)) + \frac{\mu}{4}$. Hence we are done.
\end{proof}

\begin{remark}${}$
\begin{enumerate}

\item The reader familiar with the Bing Shrinking Criterion \cite{chapman:lectures, bing_shrinking, vm:book} will note that we
have been able to set our shrinking argument in the simplest nontrivial context, where there is only one
nondegenerate point inverse and it is a cone.  In this situation, the Criterion does not simplify the argument.

\item The only property of the $G$-action used in this argument is that its orbits produce a decomposition of $Q$
into compact sets [the nondegenerate orbits and the fixed point set $F = \Fix(\sigma_n)$].  We may replace $F$
by any compact Z-set with trivial Shape (i.e., that is Cell-Like) and the nondegenerate $G$-orbits with closed
sets, $A$, such that $\diam(A)\to 0$ uniformly as  $A \to F$.
\end{enumerate}
\end{remark}

\subsection{Isotopies}

In ~\cite{Wong67}, Wong proved that each homeomorphism of $Q$ is isotopic to the identity.  Using his general idea,
we show that any two sufficiently ``small'' based-free actions of a compact group $G$ on $Q$ are isotopic through
``small'' based-free actions.

\subsubsection{The radial homeomorphism}
Let $\phi : \mathbb{R}^2\to \mathbb{R}^2$ be the \emph{radial homeomorphism} $\bar v\to c\cdot \bar v$ where $c=c(\bar
v)$ is the positive real number such that multiplication carries $\mathbb{R}\bar v \cap (\mathbb{J}\times \mathbb{J})$ to $
\mathbb{R}\bar v \cap D^2$, $D^2$ being the unit disc in the
Euclidean metric. We will show that $\phi$ is bi-Lipschitz with respect to the standard Euclidean norms. The exact
Lipschitz constant is not important for us, so will not bother trying to find sharp bounds.

Our interest in the radial homeomorphism lies in the fact that it allows us to rotate $\mathbb{J}^2$ as easily as one
can rotate $D^2$. Indeed, if we first apply $\phi$, then rotate $D^2$ say counter-clockwise through $90^\circ$,
and then apply $\phi^{-1}$, we obtain a homeomorphism of $\mathbb{J}^2$ that sends each $(x_1,x_2)\in
\mathbb{J}^2$ onto $({-}x_2,x_1)$. Moreover, this homeomorphism is isotopic to the identity on $\mathbb{J}^2$. By
abuse of terminology, we call homeomorphisms obtained in this way \emph{rotations of $\mathbb{J}^2$}. These
homeomorphisms are well-known tools in infinite-dimensional topology, see e.g. Wong~\cite[p. 151]{Wong67}.

We will first find an estimate for the absolute value of the operator norms of the Jacobians at the points where $
\phi$ and $\phi^{-1}$ are differentiable. This will allow us to apply well-known results from elementary classical
analysis.

It is well-known that for a matrix $A={{(a_{ij})}_{i=1}^m}_{j=1}^n\in \mathbb{R}^{m\times n}$, its operator norm $\| A \|$ is
bounded by $\sqrt{mn}\, \max\{|a_{ij}|: i\le m, j\le n\}$ (\cite[2.5.3]{DuistermaatKolkI},~\cite[Chapter 5]{HornJohnson13}). We will need this result for $2{\times}2$ matrices only, for which there is a simple and direct
proof which we leave as an exercise to the reader.

To simplify our notation, for $(x,y)\in \mathbb{R}^2\setminus \{(0,0)\}$, we put $d= (x^2{+}y^2)^{-\frac{1}{2}}$.

It is easy to check that $\phi$ is given by
$$
    \phi(x,y) = \begin{cases}
                (x,y) & (x=0 \text { or  } y=0),\\
                \sgn(y)\cdot(xyd, y^2d) & (x\not= 0 \,\& \, y\not= 0 \,\& \, (|x| \le |y|),\\
                \sgn(x)\cdot(x^2d, xyd) & (x\not= 0 \,\& \, y\not= 0 \,\& \,(|y| \le |x|).
               \end{cases}
$$
Consider the lines $y=x$ and $y=-x$. The complement of the union of these
lines has four open convex components $V_1,\dots, V_4$ on each of which $\phi$ is differentiable. Let $V$ be
one of those components, say
$
    V = \{(x,y)\in \mathbb{R}^2 : |x| < y\}.
$
On $V$, $\phi$ is given by $(x,y) \mapsto (xyd, y^2d)$. Put $\psi=\phi |_V$. Since $\frac{\partial}{\partial
x} d = -x d^3$ and $\frac{\partial}{\partial y} d = -y d^3$, we get

$$
    D\psi(x,y) = \left ( \begin{array}{cc}
                    yd{-}x^2yd^3 & xd{-}xy^2d^3\\
                    -xy^2d^3 & 2yd {-} y^3 d^3
                    \end{array} \right ).
$$
Since $|xd|\le 1$ and $|yd|\le 1$, the absolute value of each entry of this matrix is at most 3. By the above, for
every $(x,y)\in V$, $\|D\psi(x,y)\| \le 6$.

Analogously, $\phi^{-1}(\bar v)=t\bar v$, where multiplication by $t$ sends $\mathbb{R}\bar v \cap D^2$ to
$\mathbb{R}\bar v \cap \mathbb{I}^2$. It is easy to check that $\phi^{-1}$ is given by
$$
    \phi^{-1}(x,y) = \begin{cases}
                (x,y) & (x=0 \text{ or } y=0),\\
                \sgn(y)\cdot{\big (}\frac{ x}{yd}, \frac{1}{d}{\big )} &  (x\not= 0 \,\& \, y\not= 0 \,\& \, (|x| \le |y|),\\
                \sgn(x)\cdot{\big (}\frac{1}{d}, \frac{y}{xd}{\big )} & (x\not= 0 \,\& \, y\not= 0 \,\& \,(|y| \le |x|).
               \end{cases}
$$
We consider $\phi^{-1}$ on one of $V_1,\dots,V_4$, say
$
    W = \{(x,y)\in \mathbb{R}^2 : |y| < -x\},
$
and put $\chi = \phi^{-1}|_W$. On $W$, $\chi$ is given by $(x,y) \mapsto -{\big (}\frac{1}{d}, \frac{y}{xd}
{\big )}$, hence
$$
    D\chi(x,y) = -\left ( \begin{array}{cc}
                    xd & yd\\
                    -\frac{y}{x^2d} +yd & \frac{1}{xd} + \frac{y^2d}{x}
                    \end{array} \right ).
$$
Now since $|y|<-x$, $\frac{1}{|xd|} = \left|\frac{\sqrt{x^2+y^2}}{|x|}\right|=\left|\sqrt{1+\frac{y^2}{x^2}}\right|
<\sqrt{2}$. From this it is easily seen that the absolute value of each entry of this matrix is at most 3, hence $\|
D\chi(x,y)\| \le 6$.

By symmetry we have the same bounds in all $V_1,\dots,V_4$. Hence by~\cite[2.5.3]{DuistermaatKolkI}, on each
$V_i$, $\phi$ is bi-Lipschitz with Lipschitz constant~6. By continuity, this means that $\phi$ is bi-Lipschitz with
Lipschitz constant 6 on every closure $\overline{V}_1,\dots,\overline{V}_4$. We claim that this implies that $\phi$ is bi-Lipschitz with Lipschitz constant 6 on all of $\mathbb{R}^2$. Indeed, take arbitrary $(x,y),(a,b)\in\mathbb{R}^2$. Consider the
straight-line segment $I=\{(1{-}t)(x,y) + t(a,b) : 0\le t\le 1\}$ joining $(x,y)$ and $(a,b)$. If it lies in one of the $
\overline{V}_i$, there is nothing to prove. If not, it meets the boundaries of the sets $\overline{V}_1,\dots\overline{V}_4$ in at
most two points, say $(p,q)$ and $(r,s)$, where $(x,y) < (p,q) \le (r,s) < (a,b)$ (in the natural order of $\mathbb{I}$). But
then, by the above,
\begin{align*}
    \|\phi(x,y){-}\phi(a,b)\| &\le \|\phi(x,y){-}\phi(p,q)\| + \|\phi(p,q){-}\phi(r,s)\|+ \|\phi(r,s){-}\phi(a,b)\|\\
                            &\le 6(\|(x,y){-}(p,q)\| + \|(p,q){-}(r,s)\| + \|(r,s){-}(a,b)\|) \\
                            &= 6 \|(x,y)-(a,b)\|.
\end{align*}
The same calculation works for $\phi^{-1}$ as well, and hence we are done.

\subsubsection{Changing the norm}\label{changing}
Now we investigate the Lipschitz behavior of $\phi$ under different but equivalent norms on $\mathbb{R}^2$ that will be important later. For $i,j\ge 0$ and $(x,y)\in \mathbb{R}^2$, put
$$
    \| (x,y) \|_{i,j} = 2^{-i} |x| + 2^{-j} |y|.
$$
This is a weighted $L_1$ norm, equivalent to the Euclidean norm $\| {\cdot} \|$. Indeed, let $i \le j$ and observe that
$$
    \mbox{$\frac{2^i}{\sqrt{2}}$} (2^{-i}|x| + 2^{-j} |y|)\le \sqrt{x^2+y^2} \le 2^j (2^{-i} |x| + 2^{-j}|y|).
$$
Thus $\phi$ and $\phi^{-1}$ both have Lipschitz constants $6{\cdot} 2^{j-i} \sqrt{2}$ with respect to the norm $\|{\cdot}\|_{i,j}$.

We now consider a rotation $\rho$ of $\mathbb{R}^2$, and its conjugate $g=\phi^{-1} \circ \rho \circ \phi$. The rotation preserves lengths, hence is bi-Lipschitz with Lipschitz constant $2^{j-i}\sqrt{2}$ with respect to the norm $\|{\cdot}\|_{i,j}$ (here $i\le j$). Thus, $g$ is bi-Lipschitz as well with Lipschitz constant $72 \sqrt{2}{\cdot} 2^{3(j-i)}$ with respect to the same norm. Observe that $g$ leaves $\mathbb{J}^2$ invariant.

We conclude that rotations of $\mathbb{J}^2$ are bi-Lipschitz.

\subsubsection{Isotopies between based-free $G$-actions}
We now show that the actions constructed in Section \ref{smallactions} are isotopic through similar actions.

\begin{theorem}\label{TheoremOnIsotopies}
If $\alpha:G\times Q\to Q$ and $\beta:G\times Q\to Q$ are based-free actions of a compact group $G$ on $Q$, there is an
isotopy $\Lambda:G \times Q \times [0,1] \to Q$
from $\alpha$ to $\beta$ through based-free $G$-actions $\lambda_t$; for every $\varepsilon > 0$ there is a $\delta >0$ such that if $\alpha $ and $\beta$ are $\delta$-close to the trivial action, then they are isotopic through based-free $G$-actions that are $\varepsilon$-close to the trivial action.  If $\Fix(\alpha)=\Fix(\beta)=
\{\bar0\}$, then we may require that $\Fix(\lambda_t)=\{\bar 0\}$ for each $t\in [0,1]$.
\end{theorem}

We will first concentrate on a special case of this result which we will show later suffices.

\begin{theorem}\label{TheoremOnIsotopiesSpecialCase}
If $\alpha:G\times Q\to Q$ and $\beta:G\times Q\to Q$ are based-free actions of a compact group $G$ on $Q$ such that $\Fix(\alpha) = \Fix(\beta) =\{\overline 0\}$, there is an
isotopy $\Lambda:G \times Q \times [0,1] \to Q$
from $\alpha$ to $\beta$ through based-free $G$-actions $\lambda_t$ such that for every $t\in [0,1]$, $\Fix(\lambda_t) = \{\overline 0\}$ and $\displ(\lambda_t) \le 10^7(\displ(\alpha) + \displ(\beta))$.
\end{theorem}

The proof of this result requires several steps.

\begin{lemma} \label{JimEen}
For each $n \ge 0$, there is an isotopy $F_n=f_{n,t}:Q\times [0,1]\to Q$ with
\begin{enumerate}
\item $f_{n,0}=\id_{Q}$,
\item each $f_{n,t}$ is bi-Lipschitz with Lipschitz constant $ 10^7$,
\item for $n > 0$, $x\in Q$ and $t\in [0,1]$, $f_{n,t}(x)\equiv_n x$,\label{JimeenDeDerde}
\item for every $x\in Q$,
$$f_{n,1}(x) =
\begin{cases}
(x_2,x_1,x_4,x_3,\dots) & (n = 0).\\
(x_1,\dots,x_n,x_{n+2},x_{n+1},x_{n+4},x_{n+3},\dots) & (n > 0).
\end{cases}
$$
\end{enumerate}
\end{lemma}

\begin{proof}
We obtain $F_n$ in two steps.
Let $\tau$ be the homeomorphism of $Q_n$ that rotates each coordinate pair $\mathbb{J}_{n+2i-1}\times
\mathbb{J}_{n+2i}$ counter-clockwise through $90^{\circ}$. Thus, $\tau(x_{n+1},x_{n+2},\dots) = (-
x_{n+2},x_{n+1},-x_{n+4},x_{n+3},\dots)$.
Then $\tau$ is isotopic to the identity on $Q_n$ by an isotopy which is a rotation on each coordinate pair $
\mathbb{J}_{n+2i-1}\times \mathbb{J}_{n+2i}$. Hence by Section \ref{changing}, each level of this isotopy is bi-Lipschitz
with Lipschitz constant $72\cdot 2^3\sqrt{2}$.
Next let $\rho$ be the homeomorphism of $Q_n$ that rotates each coordinate pair
$\mathbb{J}_{n+4i-3}\times \mathbb{J}_{n+4i-1}$  through $180^{\circ}$, thus $\rho(x_{n+1},x_{n+2},\dots) = (-
x_{n+1},x_{n+2},-x_{n+3},x_{n+4},\dots)$. Then $\rho$ is isotopic to the identity on $Q_n$ by an isotopy
which is a rotation on each coordinate pair $\mathbb{J}_{n+4i-3}\times \mathbb{J}_{n+4i-1}$.  Hence by
Section \ref{changing}, each level of this isotopy is bi-Lipschitz with Lipschitz constant
$72{\cdot}2^6\sqrt{2}$.
Observe that for every $x\in Q_n$,
$$
    (\rho\circ\tau)(x_{n+1},x_{n+2},x_{n+3},x_{n+4},\dots)=(x_{n+2},x_{n+1},x_{n+4},x_{n+3},\dots).
$$
If $n > 0$, consider the product of every function considered so far with the identity function on $\mathbb{J}^n$,
thus obtaining functions defined on $Q$.
Hence in all cases, $F_n$ is given by the two isotopies previously considered, and each of its levels is bi-
Lipschitz with Lipschitz constant $72^2\cdot 2^9\cdot 2=2^{16}\cdot 3^4=5308416 < 10^7$.
\end{proof}

The idea is to isotope $\beta$ to ``$\beta\times \alpha$'' by constructing an isotopy
$\Lambda:Q\times [0,1]\to Q$ connecting them through based-free actions.  A similar construction will
``isotope $\alpha$ to $\alpha\times \beta$''. And finally, we isotope ``$\beta\times\alpha$'' to $
``\alpha\times\beta$'', which completes the job.

We first define a sequence of actions $h_n : G \times Q\to Q$, $n\ge 1$, as follows:
$h_n(g,x) = (b_1,b_2,\dots, b_n,a_1,b_{n+1},a_2,b_{n+2},a_3,\dots),$
where
\begin{enumerate}
\item    $(a_1,a_2,a_3,\dots) = \alpha (g,(x_{n+1},x_{n+3},x_{n+5},\dots))$, and
\item    $(b_1,b_2,b_3,\dots) = \beta(g,(x_1,\dots,x_{n},x_{n+2},x_{n+4},\dots))$.
\end{enumerate}
It is clear that $h_n$ is a based-free action and $\Fix(h_n) = \{\overline 0\}$.

Observe that by the definition of $h_n$, for any $x\in Q$, there exists $c\in Q$ such that for all $g\in G$,
\begin{enumerate}[resume*]
\item $x\equiv_n c$, and
\item $h_n(g,x)\equiv_n \beta(g,c)$.
\end{enumerate}
Indeed, $c= (x_1,x_x,\dots,x_n,x_{n+2},x_{n+4}, \dots)$ does the job.

The function $h_1$ is our model for ``$\beta\times\alpha$'.

\begin{lemma}\label{eerstewenen}
For every $n \ge 1$, $\disp(h_n)\le \disp(\beta)  + \displ(\alpha)$.
\end{lemma}

\begin{proof}
For every $g\in G$ and $x\in Q$, $d(h_n(g,x),x)$ is equal to
$$
    \sum_{i=1}^n 2^{-i-1} |b_i{-}x_i| + \sum_{j=1}^\infty 2^{-n-2j-1} |b_{n+j} {-} x_{n+2j}| +
    \sum_{j=1}^\infty 2^{-n-2j}|a_j {-} x_{n+2j-1}|,
$$
which is bounded by
$$
    \sum_{i=1}^n 2^{-i-1} |b_i{-}x_i| + \sum_{j=1}^\infty 2^{-n-j-1} |b_{n+j} {-} x_{n+2j}| + \sum_{j=1}^\infty
    2^{-j-1}|a_j {-} x_{n+2j-1}|.
$$
But this expression is by (1) and (2) in the definition of $h_n$ bounded by
$
    \disp(\beta)  + \displ(\alpha),
$ hence we are done.
\end{proof}

As before, set $\check f_{m,t}=\id_G\times f_{m,t}$.

\begin{corollary}\label{tweedewenen}
For all $n\ge 1$, $m \ge 0$ and $t\in [0,1]$, $\displ( f_{m,t}^{-1}\circ h_n \circ \check f_{m,t}) \le 10^7(\disp(\beta)  + \displ(\alpha))$.
\end{corollary}

\begin{proof}
Take arbitrary $g\in G$ and $x\in Q$, and put $p = f_{m,t}(x)$, respectively, $q = h_n(g,f_{m,t}(x))= h_n(g,p)$.
By Lemmas~\ref{JimEen} and~\ref{eerstewenen},
$$
\begin{array}{ll}
d(x,f_{m,t}^{-1}\circ h_n\circ\check f_{m,t}(g,x)) &= d( f_{m,t}^{-1}(p),  f_{m,t}^{-1}(q))\\
                                                    &\le 10^7 d(p,q) \\
                                                   &\le 10^7(\displ(\beta) + \displ(\alpha)),
\end{array}
$$
as desired.
\end{proof}

We will now define the isotopy $\Lambda = (\lambda_t)_{t\in [0,1]}$ connecting $\beta$ and ``$
\beta\times\alpha$'' ($= h_1$).

We put $\lambda_0 = \beta$, and for $n \ge 1$, $\lambda_{\frac{1}{n}} = h_n$.
For $n \ge 1$, let $F_n : Q \times [\frac{1}{n+1},\frac{1}{n}] \to Q$ be the isotopy given by
Lemma~\ref{JimEen} (parametrize $[\frac{1}{n+1},\frac{1}{n}]$ by $s(t)=\frac{-1}{n(n+1)}t + \frac{1}{n}$).
It will be convenient to denote $F_n$ by $(f_{n,t})_{t\in [\frac{1}{n+1},\frac{1}{n}]}$ and $\id_G\times F_n$
by $\check F_n=(\check f_{n,t})_{t\in [\frac{1}{n+1},\frac{1}{n}]}$. Define $\lambda_t$ for $t\in [\frac{1}{n+1},\frac{1}{n}]$ as
follows:
$$
    \lambda_t =  f_{n,t}^{-1}\circ h_{n+1}\circ \check f_{n,t}.
$$
Clearly,
$
    \lambda_{\frac{1}{n+1}}= h_{n+1}.
$
We claim that $\lambda_{\frac{1}{n}} = h_n$. To see this,
take arbitrary $g\in G$ and $x\in Q$. Then
$
    f_{n,\frac{1}{n}} (x) = (x_1,x_2,\dots,x_n,x_{n+2}, x_{n+1}, x_{n+4}, x_{n+3},\dots),
$
and so,
$$
    h_{n+1}\circ \check f_{n,\frac{1}{n}} (g,x) = (b_1,b_2,\dots, b_{n+1},a_1,b_{n+2},a_2,b_{n+3},a_3,\dots),
$$
where
\begin{enumerate}
\item[$(1')$]    $(a_1,a_2,a_3,\dots) = \alpha (g,(x_{n+1},x_{n+3},x_{n+5},\dots))$, and
\item[$(2')$]    $(b_1,b_2,b_3,\dots) = \beta(g,(x_1,\dots,x_{n},x_{n+2},x_{n+4},\dots))$.
\end{enumerate}
And so, indeed,
$$
    (f_{n,\frac{1}{n}})^{-1}\circ h_{n+1}\circ \check f_{n,\frac{1}{n}} (g,x) = (b_1,b_2,\dots, b_n,
    a_1,b_{n+1},a_2,b_{n+2},a_3,\dots)= h_n(x).
$$
Therefore, $\Lambda$ is well-defined and continuous on $(G\times Q)\times (0,1]$.

It remains to show that if $(t_m)$ is an arbitrary sequence in $(0,1]$ converging to 0, then the sequence $
(\lambda_{t_m})$ converges uniformly to $\beta$.
 To this end, fix $\varepsilon > 0$.
 By compactness of $G$, there exists $\delta > 0$ such that if $d(x,y) < \delta$ then for all $g\in G$, $d(\beta(g,x),\beta(g,y)) < \frac{1}{2}\varepsilon$. Hence there exists $M$ such that if $x\equiv_M y$ then $d(x,y) < \frac{\varepsilon}{2}$ and for all $g\in G$, $d(\beta(g,x),\beta(g,y))< \frac{\varepsilon}{2}$. There also
 exists $N$ such that if $m \ge N$, then $t_m < \frac{1}{M}$.
 Let $m \ge \max(N,M)$ and $g\in G$ and $x\in
 Q$ be arbitrary. There is $n\ge M$ such that $t_m \in [\frac{1}{n+1},\frac{1}{n}]$.  By (3) and (4) above, we may pick $c\in Q$ such that
 for $p=f_{n,t_m}(x)$ and $c$ we have $p\equiv_{n{+}1} c$ and
 $h_{n+1}(g,p)\equiv_{n+1} \beta(g,c)$.
 By Lemma~\ref{JimEen}(\ref{JimeenDeDerde}), $f_{n,t_m}(x)\equiv_n x$, hence $x\equiv_n c$.
 Again by Lemma~\ref{JimEen}(\ref{JimeenDeDerde}),
 $$
    \lambda_{t_m}(g,x)= f_{n,t_m}^{-1} (h_{n+1}(g,f_{n,t_m}(x)))= f_{n,t_m}^{-1} (h_{n+1}(g,p)) \equiv_n \beta(g,c).
 $$
 Now $d(\lambda_{t_m}(g,x), \beta(g,c)) < \frac{\varepsilon}{2}$ and $d(\beta(g,c),\beta(g,x)) <
\frac{\varepsilon}{2}$, so we are done by the triangle inequality.

We conclude that we can connect $\beta$ and ``$\beta\times\alpha$'' by a ``convenient'' isotopy. That is, an
isotopy through based-free actions whose fixed point sets all are equal to $\{\overline 0\}$ and with all levels having displacement at most $10^7(\disp(\beta) + \displ(\alpha))$.

Our model for ``$\beta\times\alpha$'' is the map $h_1 : G\times Q\to Q$ that we defined as follows:
$
    h_1(g,x_1,x_2,\dots) = (b_1,,a_1,b_2,a_2,b_3,a_3,\dots),
$
where
\begin{enumerate}
\item    $(a_1,a_2,a_3,\dots) = \alpha (g,(x_{2},x_{4},x_{6},\dots))$, and
\item    $(b_1,b_2,b_3,\dots) = \beta(g,(x_1,x_{3},x_{5},\dots))$.
\end{enumerate}
Similarly, the model for ``$\alpha\times\beta$'' is the map $g_1 : G\times Q\to Q$ defined by:
$
    g_1(g,x_1,x_2,\dots) = (d_1,c_1,d_2,c_2,d_3,\dots),
$
where
\begin{enumerate}
\item    $(d_1,d_2,d_3,\dots) = \alpha(g,(x_1,x_{3},x_{5},\dots))$, and
\item    $(c_1,c_2,c_3,\dots) = \beta (g,(x_{2},x_{4},x_{6},\dots))$.
\end{enumerate}
Let $F = F_0$ be the isotopy of Lemma~\ref{JimEen}, and define $S :  (G\times Q)\times [0,1]\to Q$ by
$$
    S_t = F_t^{-1}\circ h_1 \circ \check F_t.
$$
Then $S_0= h_1$ and $S_1 = g_1$. Clearly, $S$ is ``convenient'.

Since we can isotope ``$\alpha$ to $\alpha\times \beta$'' in the same way as we isotoped ``$\beta$ to $
\beta\times\alpha$', we are done with the proof of Theorem~\ref{TheoremOnIsotopiesSpecialCase}.

We now come to the proof of Theorem~\ref{TheoremOnIsotopies}. It is based on an interesting idea in Hohti and Junnila~\cite{hohti}.

\begin{proof}[Proof of Theorem~\ref{TheoremOnIsotopies}]
The first and the last statements in the theorem are consequences of
Theorem~\ref{TheoremOnIsotopiesSpecialCase} and the homogeneity of $Q$. So we are left with the $
\varepsilon-\delta$ statement.

Let $\mathcal{H}(Q)$ denote the group of all homeomorphisms of $Q$ endowed with the usual compact-open
topology. For every $p\in Q$, let $\gamma_p :  \mathcal{H}(Q)\to Q$ be the evaluation at $p$. That is, $
\gamma_p(f) = f(p)$ for $f\in \mathcal{H}(Q)$. Since $Q$ is homogeneous by Keller~\cite{keller}, the celebrated
Effros Theorem from~\cite{effros} implies that $\gamma_p$ is an open continuous surjection for every $p$ (see
also~\cite{noteeffros}).

Consider the point $\overline 0$ in $Q$.  By~\cite[2.10; Proposition 18]{bourbaki61}, there is a compact subset
$L$ of $\mathcal{H}(Q)$ such that $\gamma_{\overline 0}(L) = Q$ (the existence of such an $L$ is the main
tool in~\cite{hohti}). Put $K = L^{-1}$, and observe that for every $x\in Q$ there exists $f\in K$ such that $f(x) =
\overline 0$. Since $\mathcal{H}(Q)$ is contractible (essentially due to Wong~\cite{Wong67}, see
also~\cite{Torunczyk77} and~\cite{Ferry77}), there is a homotopy $H :  \mathcal{H}(Q)\times [0,1]\to
\mathcal{H}(Q)$ such that $H_0 = \id_{\mathcal{H}(Q)}$ and $H_1 = \id_Q$. Let $K'$ be the compact set
$H(K\times [0,1])$ of $\mathcal{H}(Q)$.

Let $\varepsilon > 0$. By compactness of $K'$ there is $\delta > 0$ such that for all $x,y\in Q$ with $d(x,y) <
\delta$, $d(f(x),f(y)) < \varepsilon$ for each $f\in K'$.

 Now let $G$ be a compact group which acts based-freely on $Q$ by the action $\beta :  G\times Q \to Q$,
 such that $\Fix(\beta)=\{y\}$ and $\displ(\beta) < \delta$. Pick $f\in K$ such that $f(y) = \overline 0$. Define an isotopy $S :  Q\times [0,1]\to Q$ by $S_t = H_{1-t}(f)$. Then $S_0 = \identity_Q$ and $S_1 = f$. For every
 $t\in [0,1]$, define an action $\gamma_t :  G\times Q\to Q$ by $\gamma_t(g,x) = S_t ( \beta(g,S_t^{-1}(x)))$.
 Then $\gamma_0 = \beta$, and each $\gamma_t$ is a based-free $G$-action with fixed point set $\{S_t(y)\}$.
 Hence $\gamma_1$ is a a based-free $G$-action with fixed point set $\{S_1(y)\} = \{f(y)\} = \{\overline 0\}$. We
 check the displacements of the various $\gamma_t$'s.

To this end, fix $t\in [0,1]$, $x\in Q$, and $g\in G$. We compare $S_t \circ \beta(g,S_t^{-1}(x))$ and $x = S_t
(S_t^{-1}(x))$. Now by assumption,
$$
    d(\beta(g,S_t^{-1}(x)), S_t^{-1}(x)) \le \displ(\beta) < \delta,
$$
and so
$$
    d(S_t(\beta(g,S_t^{-1}(x)), S_t(S_t^{-1}(x)))<\varepsilon.
$$
Hence the displacements of the $\gamma_t$'s does not exceed $\varepsilon$. Since our choice of $\delta> 0$
does not depend on $\beta$, we are done by applying the special case
Theorem~\ref{TheoremOnIsotopiesSpecialCase}.
\end{proof}

\begin{question}
Can Theorem~\ref{TheoremOnIsotopies} be generalized to semifree actions of compact groups on $Q$ whose
fixed point sets are homeomorphic Z-sets?
\end{question}

\section{From based-free to semifree actions}

We begin by noting the following which was used in the introduction.

\begin{proposition}\label{product}  Suppose that $G$ is a compact group that acts non-trivially and semifreely
on $Q$.  If $X$ is a Z-set in $Q$ that is homeomorphic to a product of fixed point sets of semifree $G$-actions
on $Q$ then there is a semifree action of $G$ on $Q$ with fixed point set $X$.
\end{proposition}

\begin{proof}  Using the Homeomorphism Extension Theorem (see Section \ref{infdim}), we may assume that $Q = \Pi_{i\ge 1}Q_i$ and
that $X = \Pi_{i\ge 1}X_i$ where each $Q_i$ is a Hilbert cube and each $X_i\subseteq Q_i$ is the fixed point set of
a semifree action $\alpha_i$ of $G$ on $Q_i$.  Now the diagonal action $\alpha$ of $G$ on $Q$ is
semifree with fixed point set $X$.
\end{proof}

Next, we extend Theorem ~\ref{point} to acyclic \ANR's and compact Lie groups.

\begin{theorem}\label{acyclicANR}
Let $G$ be a nontrivial compact Lie group.
If $A$ is an acyclic compact
\ANR\ embedded in $Q$ as a Z-set, then $A$ is the fixed point set of a semifree \ANR\ action $\alpha$ of $G$ on $Q$.
\end{theorem}

\begin{proof}
Let $Y = A\ast G$ and denote by $A'$ the image of $A\times \{-1\} \times G$ in $Y$.  It suffices to prove that $Y \in \AR\ $ since then

\begin{enumerate}

\item $A'$ is the fixed point set of the semifree \ANR\ action $\beta = \identity \ast \beta'$ on $Y$, where
$\beta'$ is the left translation action of $G$ on itself;

\item $A'\times \{v\}$ is the fixed point set of the diagonal action $\gamma = \beta \times \alpha_G$ on
$Y \times Q(G)$, where $v$ is the fixed point of the standard based-free action $\alpha_G$ of $G$ on
$Q(G)$;

\item $Y \times Q(G)$ is a Hilbert cube;

\item $A' \times \{v\}$ is a Z-set in $Y \times Q(G)$;

\item there is a homeomorphism $h:Y\times Q(G)\to Q$ such that $h(A' \times \{v\}) = A$ (by the Homeomorphism Extension Theorem); and

\item the $G$-action $\alpha$ on $Q$ induced by $h$ is a semifree \ANR\ action
with fixed point set $A$.
\end{enumerate}

Now $Y = B \,\cup \, C$ where $B = A\times C(G), C = C(A) \times G,
\text{ and } B \cap C = A\times G = D$. Because $A$ and $G$ are \ANR's, so are $B, C, \text{ and }D$.
Hence, $Y\in\ANR\ $ \cite{borsuk:retracts}.  Therefore we only need to show that it is
contractible.  Since  \ANR's have the homotopy types of CW complexes \cite{Hanner}, this follows from the Whitehead Theorem
\cite[Corollary 4.33]{Hatcher01} once we verify that it is simply-connected and acyclic in integral
homology.

First, assume that $G$ is connected.  Then the van Kampen Theorem applies to show that $Y$ is
simply-connected.  Choose a basepoint $y= (a_0,g_0) \in A \times G$.
Then $\pi_1(D,y) \cong \pi_1(B,y)\times \pi_1(C,y)$.  The inclusion homomorphisms
$i_{\ast}:\pi_1(D,y)\to \pi_1(B,y)$ and $j_{\ast}:\pi_1(D,y\to \pi_1(C,y)$ induce coordinate
projections of this product, so the inclusion $k_{\ast}:\pi_1(D,\ast)\to \pi_1(Y,\ast)$ is trivial.  Therefore
$\pi_1(Y,\ast)$, the amalgamated free product of $\pi_1(B,\ast)$ and $\pi_1(C,\ast)$, is trivial.

In integral homology, the Mayer-Vietoris theorem gives a long exact sequence
\begin{multline*}
\dots\to H_n(A\times G)\overset{\phi}\to H_n(A\times C(G))\oplus H_n(C(A)\times G)\to \\ H_n(Y)
\overset{\partial}\to H_{n-1}(A\times G)\to\dots
\end{multline*}
Since $A$ is acyclic, the K\"unneth formula \cite[Theorem 3B.6]{Hatcher01} gives
$H_n(A\times G)\cong H_n(G), H_n(A\times C(G) =0 \text{ for }n \neq 0, \text{ and }H_n(C(A)\times G)\cong H_n(G)$, so $\phi$ is an isomorphism and $Y$ is acyclic.

In the case that $G$ is not connected, the maximal connected subgroup $G_0$ of the identity is closed and normal and we
have the exact sequence $1\to G_0\to G\to F\to 1$, where $F$ is finite.  Thus, as a topological space,
$G \approx G_0 \times F$ is the disjoint union of $|F|$ copies of $G_0$, one for each element of $F$,
and $Y$ is the union of $|F|$ copies of $A\ast G_0$ identified along their common copy of $A'$.  An induction on the cardinality of $F$ shows that $Y$ is contractible.
\end{proof}

\begin{theorem}\label{semifreeapprox} Let $G$ be a compact group.  If there is a nontrivial semifree action $\alpha$ of $G$ on a Keller space with
fixed point set $A$, then for each Keller space $X$
\begin{enumerate}

\item  there is a semifree $G$-action $\beta$ on $X$
with fixed point set a Z-set homeomorphic to $A$,

\item if $\gamma$ is a semifree action of $G$ on $X$ with fixed point set homeomorphic to $A$, $\beta$ may be taken to be arbitrarily close to $\gamma$, and

\item if $\alpha$ is a $G$-\ANR\ action, $\beta$ may be required to be a $G$-\ANR\ action.
\end{enumerate} \end{theorem}

\begin{proof}  Since all Keller spaces are Hilbert cubes, we may assume that $X = Q$.  Let $\alpha$ be a
semifree action of $G$ on $Q$ with fixed point set $A$.  By Theorem \ref{point}, there is a based-free action $
\delta$ of $G$ on $Q$.  Let $x_0$ be the fixed point of $\delta$.  Then the diagonal action $\varepsilon =
(\alpha,\delta)$ on $Q\times Q$
is semifree with fixed point set $A\times \{x_0\}$, which is a Z-set.  By manipulating the indexing of the
coordinates, we may approximate the projection of $Q\times Q$ onto the first factor by a homeomorphism as
closely as desired.
\end{proof}

We now generalize Theorem~\ref{derdestelling} to semifree actions with nontrivial fixed point sets. In the proof
of this generalization, we need the full strength of Theorem~\ref{derdestelling}.

\begin{theorem}\label{derdestellingPartTwo}
Let $\alpha$ be a semifree action of a compact group $G$ on $Q$ such that $\Fix(\alpha)$ is a Z-set.
There is a semifree action $\beta$ of $G$ on $Q$ arbitrarily close to the trivial action such that $
\Fix(\alpha)=\Fix(\beta)$.
\end{theorem}

\begin{proof}
We prove this assuming that $\Fix(\alpha)$ is a subset of $\Pi_{n=1}^{\infty}(-1,1)_n$. By the Estimated Homeomorphism Extension
(see Section \ref{infdim}), this suffices.

We consider $Q\times Q$ with projections $p_1(x,y) = x$ and $p_2(x,y) = y$.

\begin{claim}
There is a homeomorphism $f :  Q\times Q \to Q\times Q$ such that
\begin{enumerate}
\item for $x\in \Fix(\alpha)$, $f(x,x) = (x,\overline 0)$,
\item for $x,y\in Q$, $p_1(f(x,y)) = x$,
\item for $x,y\in Q$ and $g\in G$, $p_2(f^{-1}(x,y))= p_2(f^{-1}(\alpha(g,x),y))$.
\end{enumerate}
\end{claim}
Assume that $\Fix(\alpha)\subseteq \prod_{n=1}^\infty [-r_n,r_n]$, where for each $n$, $0 < r_n < 1$.
We will show that
there is a continuous function $p :  Q\to \prod_{n=1}^\infty [-r_n,r_n]$ such that $p$ restricts to the identity
on $\Fix(\alpha)$, while moreover for every $x\in Q$ and $g\in G$, $p(x) = p(\alpha(g,x))$.

To prove this, consider the orbit space $Q/G$ with standard quotient map $\kappa:  Q\to Q/G$.  We identify $
\Fix(\alpha)$ and $\kappa(\Fix(\alpha))$. By the Tietze Extension Theorem, the identity $
\kappa(\Fix(\alpha))\to \Fix(\alpha)$ can be extended to a continuous function $u :  Q/G \to
\prod_{n=1}^\infty [-r_n,r_n]$. Hence $p = u \circ \kappa$ is as required.
\par\medskip

For every $t\in (-1,1)$, let $\phi_t :  \mathbb{J}\to \mathbb{J}$ be the unique homeomorphism taking the interval
$[-1,t]$ linearly onto $[-1,0]$ and $[t,1]$ linearly onto $[0,1]$.  For each $q\in \prod_{n=1}^\infty [-r_n,r_n]$
define a homeomorphism $g_q :  Q\to Q$ by
$$
    g_q(x) = (\phi_{q_1}(x_1), \phi_{q_2}(x_2), \dots, \phi_{q_n}(x_n), \dots ).
$$
Now define $f$ by $f(x,y) = (x,g_{p(x)}(y))$. Then $f$ is a homeomorphism. If $x\in \Fix(\alpha)$, then
$$
    f(x,x) = (x,(\phi_{x_1}(x_1), \phi_{x_2}(x_2), \dots, \phi_{x_n}(x_n), \dots )) = (x,\overline 0).
$$
Now take arbitrary $x,y\in Q$. Put $q=p(x) = p(\alpha(g,x))$. Then $f^{-1}(x,y) = (x,g_q^{-1}(y))$ and $f^{-1}
(\alpha(g,x),y) = (\alpha(g,x), g_q^{-1}(y))$, and so (3) is obvious.

\par\medskip

For notational convenience, let $A=\Fix(\alpha)$. Consider the diagonal embedding $ \iota:A \to\triangle A$ of
$A$
in $Q^{\infty}$. If we can prove that $\triangle A$ is the fixed point set of $G$-actions arbitrarily close to the
trivial action on $Q^\infty$, then trivially so is any Z-set copy of $A$ in $Q$ (again by the Estimated
Homeomorphism Extension Theorem (see Section \ref{infdim})).

Let $n$ be arbitrary for the time being, and consider $A_n= \{\bar 0\} \times A \times \{\bar 0\}$, which is $A$ in
the $n$-th $Q $ coordinate. By Claim~1, there is a homeomorphism $f_n : Q^\infty\to Q^\infty$ such that

\begin{enumerate}
\item $f_n(\triangle A) = A_n$,
\item for each $x\in Q^\infty$, $(f_n(x))_n = x_n$,
\item if $x,y\in Q^\infty$ are such that $x_i = y_i$ for all $i \not= n$ and there is a $g\in G$ such that
$y_n=\alpha(g,x_n)$ , then $(f_n^{-1}(x))_i= (f_n^{-1}(y))_i$ for all $i \not = n$.
\end{enumerate}

Now fix $\varepsilon > 0$, and pick $N$ so large that $\sum_{n=N}^\infty 2^{-n} < \frac{1}{2}\varepsilon$.
Consider $f_N$, and pick $\delta > 0$ such that if $x,y\in Q^\infty$ are such that $\varrho(x,y) < \delta$, then $
\varrho(f_N^{-1}(x),f_N^{-1}(y)) < \frac{1}{2}\varepsilon$.

By Theorem~\ref{derdestelling} and the homogeneity of $Q$, there is an action $\tau :  G\times Q\to Q$
such that $\Fix(\tau) = \{\overline 0\}$ while moreover for every $g\in G$, the homeomorphism $x\mapsto
\tau(g,x)$ has displacement less than $\delta$. Define an action $\beta :  G\times Q^\infty \to Q^\infty$
coordinatewise, as follows:
$$
    \beta(g,x)_n = \begin{cases}
                  \tau(g,x_n) & (n\not= N), \\
                  \alpha(g,x_n) & (n=N).
                  \end{cases}
$$
Observe that $\Fix(\beta) = A_N$.

Define $\gamma :  G\times Q\to Q$ by $\gamma(g,x) = f^{-1}_N (\beta(g,f_N(x)))$. Clearly, $\Fix(\gamma) =
\triangle A$.

Take an arbitrary $x\in Q^\infty$, and let $y = f_N(x)$. By (2), $y_N = x_N$. Hence for a fixed $g\in G$,
$$
    \beta(g,y) = (\tau(g,y_1),\dots,\tau_{N-1}(g,y_{N-1}), \alpha(g,x_N), \tau(g,y_{N+1}), \tau(g,y_{N+2}), \dots).
$$
Consider
$$
    z = (\tau(g,y_1),\dots,\tau_{N-1}(g,y_{N-1}), x_N, \tau(g,y_{N+1}), \tau(g,y_{N+2}), \dots).
$$
Then, clearly, $\varrho(y,z) < \delta$, hence $\varrho(x,f_N^{-1}(z)) < \frac{1}{2}\varepsilon$. Now by (3), $(f_N^{-1}(\beta(g,y)))_i= (f_N^{-1}(z))_i$ for all $i \not= N$ and so, by the choice of $N$, $\varrho( f_N^{-1}(\beta(g,y)), f_N^{-1}(z)) < \frac{1}{2}\varepsilon$. Hence $\varrho(x,\gamma(g,x)) < \varepsilon$.
\end{proof}

\section{Applications: \ANR\ Shapes and another Converse to P.\ A.\ Smith's Theorem}\label{fixed pointsets}

\subsection{Fixed-point sets and Shape equivalences}

We will now apply the results obtained in the settings of Shapes of fixed point sets. For
information on Shape theory, see Section 7.
By $\Sh(X) = \Sh(A)$ we mean that $X$ and $A$ are isomorphic in Borsuk's Shape category.
When $A$ is a compact \ANR, this topic is discussed in Section 7.  For compact \ANR's,
Shape equivalence is the same as homotopy equivalence.

In the following result, we will use the celebrated Chapman Complement Theorem \cite[25.2]
{chapman:lectures}.
Note that here if $X$ is not an \ANR\, then the action cannot be an equivariant \ANR.

\begin{theorem}\label{Shape}  Let $G$ be a compact group.  If $A$ is a compact
\ANR\ that is the fixed point set of a semifree action of $G$ on $Q$, then every Z-set $X$ of  $Q$ that is
Shape equivalent to $A$ is the fixed point set of a semifree action of $G$ on $Q$.  Moreover, we can choose this action to be arbitrarily close to the trivial action.

\end{theorem}

\begin{proof}
We will first show that we may assume without loss of generality that we deal with actions on $K=C(A)\times
Q$.  As in the proof of
Theorem~\ref{derdestelling}, our use of cones simplifies the technical
details of the construction.

Let $\alpha$ be a semifree action of $G$ on $Q$ with fixed point set $A$.  By Theorem~\ref{point} there is a based-free action $\sigma$ of $G$ on $Q$.  As $Q$ is homogeneous, we may assume that the fixed point of $\sigma$ is $\{\bar 0\}$. Define a semifree action on $Q\times Q\times Q$ by $
(g,x,y,z)\mapsto (\alpha(g,x),y, \sigma (g,z))$.  The fixed point set of this action is
$A\times Q\times \{\bar 0\}$, which is a Z-set in $Q\times Q \times Q$.  Because
$C(A)$ is an \AR,  the product $K=C(A) \times Q\approx Q$ by Edwards' Theorem (see Section \ref{infdim}).  Also, $(A\times \{\bar 0\})\times Q$ is a Z-set in $K$.
By the Homeomorphism Extension Theorem (see Section \ref{infdim}), there is a homeomorphism of
$Q\times Q\times Q\to K$ carrying
$A\times Q\times \{\bar 0\}$ in $Q\times Q\times Q$ to $(A\times \{\bar 0\})\times Q$ in $K$. Therefore we may
indeed assume without loss of
generality that $\alpha$ acts semifreely on $K$ with Fix($\alpha)=(A\times \{\bar 0\})\times Q$ and that $X$ is a
Z-set in $Q$.

We use the metric $\varrho$ on $K$ given by
$$
    \varrho(((x,s),u),((y,t),v)) = (1-\max\{s,t\}) d(x,y) + |t{-}s| + d(u,v).
$$
(See Section \ref{preliminaries} for this notation).
Denote $A\times [0,t]\times Q$ by $K_t$.  Then Fix($\alpha)= K_0$.  Since $\Sh(K_0)=\Sh(X)$, there is by
the Chapman Complement Theorem \cite[25.2]{chapman:lectures}, a homeomorphism $f: K\setminus K_0\to
Q\setminus X$.
For every $i \ge 1$, let $c_i = 2^{-i}$. Then for each $i$, $P_i=\{((x,s),u) \in K : s \ge c_i\}$ is a compact subset
of $K\setminus K_0$. Hence by uniform continuity, we may choose a sequence $1 > \delta_1 > \delta_2 >\dots
> 0$ such that

\begin{enumerate}
\item if $w,z\in P_i$ and $\varrho(w,z) < \delta_i$, then $d(f(w),f(z)) < \frac{1}{i}$.\label{nieuw2}
\end{enumerate}

Fix a
decreasing sequence $1 > d_1 > d_2 >\dots>0$ such that

\begin{enumerate}[resume*]
\item for $i \ge 1$, $d_{i} < c_{i+1}$, \label{nieuw1}
\item the displacement of $\alpha$ on $K_{d_i}$ is less than $\frac{\delta_i}{3}$, and, \label{deeerste}
\item $\alpha (G\times K_{d_{i+1}})\subseteq K_{d_i}$.\label{detweede}
\end{enumerate}

For each $i\ge 1$ let
$$
    c_{i+1} = t_{k(i)}^i < t_{k(i)-1}^i < \cdots < t_2^i < t_1^i = c_i
$$
be a partition of $[c_{i+1},c_i]=[2^{-(i+1)}, 2^{-i}]$ such that

\begin{enumerate}[resume*]
\item $k(i) > 1$ and each subinterval
$[t_{m+1}^i,t_m^i], m\le k(i){-}1$, has length less than $\frac{\delta_{i+1}}{6}$. \label{hellupvier}
\end{enumerate}
Now let $j(1)=3$ and for $i\ge 1$ set $ j(i{+}1)=j(i)+k(i)-1$. Clearly,
\begin{enumerate}[resume*]
\item $j(i{+}1) > j(i) \ge i{+}2$. \label{afschatting}
\end{enumerate}

Let $\phi:[0,1]\to [0,1]$ be a homeomorphism
such that for each $i \ge 1$,

\begin{enumerate}[resume*]
\item $\phi(c_i)=d_{j(i)}$, \label{devierde}
\item for $m \le k(i)$, $\phi(t_m^i) = d_{j(i) + m-1}$.\label{devijfde}
\end{enumerate}
Note that (\ref{devierde}) and (\ref{devijfde}) are consistent since $\phi(t_{k(i)}^i) = d_{j(i) + k(i)-1}=d_{j(i+1)}=\phi(c_{i+1})$.

Define $h:K\to K$ by $((x,t),u)\mapsto ((x,\phi(t)),u)$ and $\beta: G\times K \to K$ by $\beta =(h^{-1}\circ \alpha) \circ (\identity_G\times h)$. Hence for $((x,t),u) \in K$ and $g\in G$,
$$
    \beta(g,((x,t),u))= h^{-1}(\alpha(g,(x,\phi(t)),u)).
$$
For $z=((x,t),u)\in K_{c_i}$, $g\in G$, and $m\le k(i){-}1$ such that
\begin{enumerate}[resume*]
\item $\cdots < t^{i+1}_1=c_{i+1}=t_{k(i)}^i \le t_{m+1}^i < t\le t_m^i \le t_{1}^i=c_i=t^{i-1}_{k(i-1)} < \cdots$, \label{dezevende}
\end{enumerate}
let
$w = \beta (g,z)=h^{-1}(\alpha(g,h(z)))=((y,s),v)$. Also, let $\nu= j(i) + (m{-}1)$.  Then $\phi(t^i_m) = d_\nu$, and $\phi(t^i_{m+1}) = d_{\nu +1}$.

Set $t' = \phi(t)$ and $s' = \phi(s)$. Hence,

\begin{enumerate}[resume*]
\item $z= ((x,t),u)$,\\ $h(z) = ((x,t'),u)$,\\ $w= \beta(g,z)=h^{-1}(\alpha(g,h(z)))= ((y,s),v)$, and\\ $h(w) = \alpha(g,h(z))= ((y,s'),v)$. \label{opsomming}
\end{enumerate}

We will compute how far $w$ and $z$ are apart, and what that tells us about the distance between $f(z)$ and $f(w)$.
The following diagram is helpful to see what is going on.
$$
    \xymatrix{
    \cdots \ar[d]_{\phi} t^{i+1}_1= c_{i+1} =t_{k(i)}^i  &t_{m+1}^i \ar[d]  &  \ar[d] t & t_m^i \ar[d]  & t^i_1=c_i=t^{i-1}_{k(i-1)} \cdots \ar[d]^{\phi}\\
   d_{j(i+1)}  &  d_{\nu+1} & t'        & d_\nu & d_{j(i)}
   }
$$

Put $\mu=\varrho(z,w)$. Then,
\begin{enumerate}[resume*]
\item $\mu =(1-\max \{t, s\})d(x,y)+|t-s|+d(u,v)$.
\end{enumerate}

\setcounter{claim}{0}
\begin{claim}\label{eersteclaimjawel}
$d(u,v) \le \textstyle \frac{\delta_{j(i)}}{3} -(1{-}\max\{s',t'\})d(x,y)$.
\end{claim}

\noindent Since $t \le c_i$, by (\ref{devierde}), $h(z) \in K_{d_{j(i)}}$, hence by (\ref{deeerste}), $\varrho(h(z),\alpha (g,h(z)))\le \frac{\delta_{j(i)}}{3}$.
So
$$
\varrho(h(z),h(w)) = (1{-}\max\{t',s'\})d(x,y) + |s'{-}t'| + d(u,v) \le \textstyle\frac{\delta_{j(i)}}{3},
$$
which clearly proves the claim.

\begin{claim}\label{tweedeclaimjawel}
$h(z) \in K_{d_\nu}\setminus K_{d_{\nu+1}}$ and $h(w) \in K_{d_{\nu-1}}\setminus K_{d_{\nu+2}}$.
\end{claim}

By (\ref{dezevende}), $d_{\nu+ 1} < t' \le d_{\nu}$, hence $h(z) \in K_{d_{\nu}}\setminus K_{d_{\nu+1}}$.  So by (\ref{detweede}), $h(w)=\alpha(g,h(z)) \in K_{d_{\nu-1}}$. If $h(w)\in K_{d_{\nu+2}}$, then, again by (\ref{detweede}),
$h(z) = \alpha (g^{-1}, h(w)) \in K_{d_{\nu+1}}$, which is a contradiction. Hence the claim is proved.
\par\medskip

So we ``roughly'' know where to find $h(w)$.

\begin{claim}
If $c_{i+1} \le s \le c_i$, then $d(f(z),f(w)) < \frac{1}{i+1}$.
\end{claim}

\noindent Assume first that $c_{i+1} \le s \le t$, hence $d_{j(i+1)} \le s'\le t'$. Since by Claim~\ref{tweedeclaimjawel}, $d_{\nu+2} < s'$, there are two cases. The first case is that $d_{\nu+1} \le s'$. Then $t^i_{m+1} \le s \le t\le t^i_m$, and so by (\ref{hellupvier}), $t{-}s \le \frac{\delta_{i+1}}{6}$. The second case is that $d_{\nu+2} < s' < d_{\nu + 1}$. Since $\phi(c_{i+1})=d_{j(i+1)}\le s'$, we have $j(i{+}1) > \nu{+}1$, hence $j(i{+}1) \ge \nu{+}2$. Therefore $c_{i+1} \le t^i_{m+2} < s < t^i_{m+1} < t \le t^i_{m}$. Hence in this case, again by (\ref{hellupvier}), $t-s \le \frac{2\delta_{i+1}}{6}$. This gives us by (\ref{afschatting}) and Claim~\ref{eersteclaimjawel} that
\begin{align*}
\varrho(z,w) &= (1{-}\max\{s,t\})d(x,y) + |s{-}t| + d(u,v)\\
             &\le (1{-}t) d(x,y)   + \textstyle\frac{2\delta_{i+1}}{6} + \textstyle \frac{\delta_{j(i)}}{3} -(1{-}t')d(x,y)\\
             &< (t'{-} t) d(x,y) + \delta_{i+1}.
\end{align*}
Now by (\ref{afschatting}) and (\ref{nieuw1}), $t'\le d_{\nu}\le d_{j(i)}\le d_{i+2} < c_{i+1}< t$, and so $t'{-}t < 0$. We conclude that $\varrho(z,w) < \delta_{i+1}$, and since both $z$ and $w$ are in $P_{i+1}$, we get by (\ref{nieuw2}), $d(f(z),f(w)) < \frac{1}{i+1}$.

Assume next that $t \le s \le c_i$. But this case is completely analogous to the previous one, hence we are done.

\begin{claim}\label{vierdeclaim}
If $s > c_{i}$, then $d(f(z),f(w)) < \frac{3}{i}$.
\end{claim}

By Claim~\ref{tweedeclaimjawel}, $h(w) \in K_{d_{\nu-1}}$. Since $s > c_i$, $\nu {-} 1 < j(i)$. Also, $j(i)\le \nu$. Hence $j(i) = \nu$.

Put
$$
    z'= ((x,c_i),u), \mbox{ and } z''= ((y,c_i),v).
$$
Then by (\ref{hellupvier}), $\varrho(z,z') = |t-c_i| \le \frac{\delta_{i+1}}{6}$. Since $z,z'\in P_{i+1}$, by (\ref{nieuw2}), $d(f(z),f(z'))< \frac{1}{i+1}$. Also, by Claim~\ref{eersteclaimjawel},
\begin{align*}
\varrho(z', z'') &= (1{-}c_i) d(x,y) + d(u,v)\\
                &\le (1{-}c_i) d(x,y) + \textstyle\frac{\delta_{j(i)}}{3} - (1{-}s')d(x,y)\\
                &=(s'{-}c_i) d(x,y) + \textstyle\frac{\delta_{j(i)}}{3}.
\end{align*}
Now by (\ref{afschatting}) and (\ref{nieuw1}), $s'\le d_{\nu - 1}= d_{j(i)-1}\le d_{i-1} < c_{i}$, and so $s'{-}c_i < 0$. We conclude that $\varrho(z,'z'') < \delta_{i}$, and since both $z'$ and $z''$ are in $P_{i}$, we get by (\ref{nieuw2}), $d(f(z'),f(z'')) < \frac{1}{i}$. Finally,
$\varrho(z'',w) = |c_i-s|\le \frac{\delta_i}{6}$. Hence, again by (\ref{nieuw2}), $d(f(z''), f(w)) < \frac{1}{i}$. We conclude that
\begin{align*}
    d(f(z),f(w)) &\le d(f(z),f(z')) + d(f(z'), f(z'')) + d(f(z''),f(w)) \\
    &< \textstyle\frac{1}{i+1} + \textstyle\frac{1}{i} + \textstyle\frac{1}{i}\\
    &< \textstyle\frac{3}{i}.
\end{align*}

So we are left with the case that $s < c_{i+1}$. But this case is completely analogous to the previous case (with bound $\frac{3}{i+1}$).

Thus we see that the action $\gamma:G\times Q\setminus X\to Q\setminus X$ given by
$\gamma(g,w)=f(\beta (g,f^{-1}(w)))$ extends to a semifree action on $Q$ with fixed point set $X$.

For the final statement in the theorem, just apply Theorem~\ref{derdestellingPartTwo}.
\end{proof}

\begin{corollary} Let $X \subset Q$ be a Z-set in $Q$ that is acyclic in integral \v Cech homology.
If $X$ is shape equivalent to an \ANR\ then every nontrivial compact Lie group $G$ acts semifreely
on $Q$ with fixed point set $X$ and arbitrarily small displacement.
\end {corollary}

\begin{proof} Let $A\in $ \ANR\ be shape equivalent to $X$.  Since \v Cech homology is a shape invariant, $A$ is acyclic in integral \v Cech and hence singular homology.  Thus, if we embed
$A$ in $Q$ as a Z-set, Theorem \ref{acyclicANR} produces a semifree $G$-action on $Q$ with fixed point set $A$.  By Theorem \ref{Shape}, $X$ is the fixed point set of semifree $G$-actions arbitrarily close to the trivial action.
\end{proof}

\begin{corollary}\label{CE}
Let $X$ be a Cell-like Z-set in a Keller space $K$.  If $G$ is a nontrivial compact Lie group, there are semifree actions of $G$ on $K$ with fixed point set $X$ arbitrarily close to the identity.
\end{corollary}

\subsection{A Partial Converse to P.\ A.\ Smith's Theorem}

Here we interpret our results as providing a partial converse to the restriction of P.A. Smith's theorem on periodic maps \cite{Smith41} in the context of Keller spaces.  They may be viewed as an extension of the Theorem of L. Jones mentioned in the Introduction.
    As usual, for an integer $n$, $\mathbb{F}_n$ denotes the field $\mathbb{Z}/n\mathbb{Z}$. A space $X$ is
\emph{acyclic mod $n$} if $\Check{H}_k(X) = \mathbb{F}_n$ if $k=0$ and 0 otherwise ($\Check{H}$ denotes
\v{C}ech homology). Hence $X$ is compact, connected, and nonempty, that is, a continuum.   A periodic homeomorphism of period $n$ is  {\em semifree} if the action of $\mathbb Z_n$ it induces is semifree.

\begin{theorem}\label{stellingoverANRs}
Let $X$ be a compact \ANR. Then $X$ is the fixed point set of a periodic homeomorphism of $Q$ of period $n$ acting semifreely if and only if $X$ is acyclic mod $n$.
\end{theorem}

\begin{proof}
There remains to show that if $X$ is acyclic mod $n$, then it is a fixed point set of a semifree homeomorphism of period $n$. Since
$X$ is a compact \ANR, $X\times Q$ is a compact $Q$-manifold by Edwards' Theorem (see Section \ref{infdim}),
cf.~\cite{Tor:hilbertcube}. There consequently is by Chapman~\cite[36.2]{chapman:lectures} a compact
polyhedron $P$ such that $X\times Q\approx P\times Q$. Clearly, $P$ is acyclic mod~$n$. Hence there exists
by Jones~\cite[1.1]{Jones71} a contractible compact polyhedron $L$ containing $P$ and a periodic
semifree homeomorphism
$\alpha :  L\to L$ of period $n$ such that $\Fix(\alpha)=P$.  Let $\tau$ be any semifree homeomorphism of $Q$ of
period $n$ with $\Fix(\tau) = \{\overline 0\}$ (for example, a simple product of rotations).  Consider the space $S=L\times
Q\times Q$ with its period $n$ homeomorphism
$\beta = \alpha \times \tau\times \mathrm{id}_Q$.  By \cite[5.6]{West70} (or Edwards' Theorem), $S\approx Q$. Observe that $\Fix(\beta) = P\times \{\bar 0\}\times
Q$ is a Z-set in $S$, so that the natural homeomorphism $P\times \{\bar 0\}\times Q\to P\times \{\bar 0\}
\times Q\times \{0\}$ extends to a homeomorphism of $S$ onto $P\times \{\bar 0\}\times Q\times [0,1]$.  Since
$P\times Q\approx X\times Q$, there is a cell-like map
$f :  \Fix(\beta)
\to X$. The adjunction space $T = S\cup_f X$ is thus homeomorphic to the mapping cylinder of $f$ and is
homeomorphic to $Q$ by \cite[Theorem 2.4]{west:borsuk} (see also~\cite[Theorem 3]{Tor:hilbertcube}).  Now $f$
induces a period $n$ homeomorphism on the Hilbert cube $T$ with fixed point set $X$.
\end{proof}

The following together with Corollary \ref{CE}, condenses our main results on Keller spaces.  Observe that
every compact \ANR\ has the Shape of a compact polyhedron by
\cite{west:borsuk}. Hence the class of compact spaces having the Shape of some
compact \ANR\ coincides with the class of compact spaces having the Shape of some
compact polyhedron.

\begin{corollary}\label{shapeanr:p-acyclic}
Let $X$ be a Z-set with the Shape of a compact \ANR\ in a Keller space $K$. Then for every $n >1$, the following statements are equivalent:

\begin{enumerate}

\item $X$ is acyclic mod $n$.

\item There is a homeomorphism of $K$ of period $n$ with fixed point set $X$.

\item Arbitrarily close to the identity there are semifree homeomorphisms of $K$ of period~$n$ with fixed point set $X$.
\end{enumerate}
\end{corollary}

\begin{proof}
Since being acyclic mod $n$ is a Shape invariant of metric compacta (\cite[page 122, Corollary 1]{mardseg}), we are done by applying Theorems~\ref{stellingoverANRs}, and \ref{Shape}.
\end{proof}

\section{Compact Infinite-Dimensional Manifolds}\label{manifolds}

Let us recall that a manifold modeled on a Keller space $X$ is a locally compact space in
which each point has an open neighborhood
homeomorphic to an open subset of $X$, i.e., to an open subset of the Hilbert cube $Q$.  Since $Q$
is homogeneous and by Keller's theorem is homeomorphic to its cone, it follows immediately from the Homeomorphism Extension Theorem that every point has
arbitrarily small open neighborhoods homeomorphic to $Q\times [0,1)$.
We now apply some results of Cappell, Weinberger, and Yan~\cite{cwy} to $Q$-manifolds.
{\em Unless otherwise noted, for the rest of this paper, $G$ and
$M$ will denote a finite group and a
connected compact $Q$-manifold, respectively, and $\tilde M$ will always denote the universal
cover of $M$.  We adopt the convention of \cite{cwy} in which if $p:\tilde B \to B$ is a covering projection and $f:A\to B$,
then $\tilde A$ denotes the pull-back $f^{\ast}\tilde B$ and $\tilde f = f_{\ast}:\tilde A\to \tilde B$
is the canonical map covering $f$. }  All results apply with appropriate modifications to the case that $M$ is not
connected.

We use the following definitions.

\begin{definition} ${}$ \begin{enumerate}

\item If $X$ is a space, $X^Q = X\times Q$ is its {\em stabilization}.  If $f:X\to Y$ is a map, $f^Q=f\times \identity_Q:X^Q\to Y^Q$ is its stabilization. If $\alpha:H\times X\to X$ is an action of $H$ on $X$, then $\alpha^Q:H\times X^Q\to X^Q$ by $\alpha^Q(h,(x,y))=(\alpha(h,x),y)$ is its stabilization.

\item A {\em triangulation} of a $Q$-manifold $N$ is a homeomorphism $h:K\times Q\to N$, where $K$ is a simplicial complex.
 \end{enumerate}
\end{definition}

Chapman proved \cite [Lemma 37.1]{chapman:lectures} that each pair
$(N,A)$ of $Q$-manifolds with $A$ a Z-set in $N$ has a triangulation of pairs $h:(K^Q,L^Q)\to (N,A)$  for a
pair of simplicial complexes $(K,L)$ (with $L$ a subcomplex of $K$) and that every triangulation of $A$
extends to one of the pair $(N,A)$.  Note that if $f:K^Q\approx N$ and $g:L^Q\approx N$ are two
triangulations
of a $Q$-manifold $N$, then $K\overset{\times \bar 0}\longrightarrow K^Q\overset{ f}\to N\overset{g^{-1}}
\to L^Q\overset{p}\to L$ is a homotopy equivalence, where $p$ is the projection.  In fact, it is a simple
homotopy equivalence \cite{chapman:torsion}. For simple homotopy, Whitehead groups $Wh(\pi)$, and
Whitehead torsion, see \cite{cohen} and \cite{milnor:wht}.  Together with the proof by the second author
\cite[proof of Theorem 5.2]{West70} that stabilizations of simple homotopy equivalences between finite
simplicial complexes are homotopic to homeomorphisms, this allows the extension
\cite{chapman:simpleforANR's} of simple homotopy theory to compact connected $Q$-manifolds by the
definition, for a homotopy equivalence $f:M\to N$, of
the torsion $\tau(f) = \tau(\hat f)\in Wh(\pi_1(N)),\text{ where }\hat f : K\overset{\times \{\bar 0\}}\to
K^Q\approx M \overset {f}\to N \approx L^Q\overset{p}\to L$.  Then $\tau(f) = 0$ if and only if $f $ is
homotopic to a homeomorphism.  Edwards' Theorem permits the further extension to compact connected
\ANR's in \cite{chapman:simpleforANR's} by the same method.

The stabilization of a cellular action of $G$ (on a finite cell complex) is a $G$-\ANR\ in the
category of $G$-actions \cite[Theorem 4.2]{jaworowski}, and in a $G$-\ANR\ action the fixed point set
of each subgroup must be an (inequivariant) \ANR, so
Theorem~\ref{Shape} and its generalization in the following Section show that all
nontrivial non-free
semifree $G$-actions on $Q$-manifolds may
be approximated by non-\ANR\  semifree actions with non-\ANR\ fixed point sets.  Thus, a $G$-action on a compact $Q$-manifold may
be conjugate to the stabilization of a
cellular action on a finite cell complex, the stabilization of an \ANR\ or non-\ANR\ action on a
compact \ANR\, or not conjugate to the
stabilization of any $G$-action on a compact \ANR.  Anderson's Conjecture is equivalent to the statement that for $G = \mathbb Z_2$ all based-free actions on $Q$ are \ANR\ actions  \cite{WestWong79}.

We next collect several facts for reference.

\begin{proposition}\label{collection} Let $H$ be a compact
group.
\begin{enumerate}
\item If a Z-set of the form $A\times Q$ is the fixed point set of a semifree $H$-action on $M$ then so is $A \times \{\bar 0\}$.

\item If $A$ is the fixed point set of a nontrivial semifree $H$-\ANR\ action $\alpha$ on $M$, then
\begin{enumerate}

\item $A$ is an (inequivariant) \ANR,

\item $A^Q$ is a Hilbert cube manifold,

\item $\alpha$ may be approximated  arbitrarily closely by semifree $H$-\ANR\  actions
$ \beta$ and $\gamma$ on $M$ with Z-set fixed point sets
$A'$ and $(A^Q)'$, respectively, such that

\begin{enumerate}

\item there is a homeomorphism $\zeta:A'\to A$ arbitrarily close to the inclusion $A'\to M$, and

\item there is a homeomorphism $\eta:(A^Q)'\to A^Q$ such that the map $(A^Q)'\overset{\eta}{\to} A^Q\overset
{p}{\to} A$ is arbitrarily close to the inclusion $(A^Q)'\to M$, where $p:M^Q\to M$ is
the projection.
\end{enumerate}
\end{enumerate}
\end{enumerate}
\end{proposition}

\begin{proof} If a Z-set $A\times Q$ is the fixed point set of the $H$-action $\alpha$, let $f_t:A\times Q\to A\times Q $ be given by $(a,y)\mapsto (a,t{\cdot} y)$ for $0 < t\le 1$.  The Estimated Homeomorphism
Extension Theorem (see Section \ref{infdim}) shows that for each open cover $\mathcal
U$ of $M$ such that every $\{a\}\times Q$ is in some member $U(a)$ of $\mathcal U$, there is a homeomorphism $\bar f_t$ of $M$ extending $f_t$ such that for each $x\in M$
there is a $U_x\in \mathcal U$ containing $x$ and $\bar f_t(x)$ and for each $a\in A$,
$\bar f_t(\{a\}\times Q) = \{a\}\times tQ$.  The Bing Shrinking Criterion \cite{bing_shrinking}
then applies to show that the quotient map $\kappa:M\to M/\sim$, where
$\sim $ is the equivalence relation generated by $(a,y)\sim (a,z)$ for $a\in A$ and $y$ and $z$ in $Q$,
is a uniform limit of homeomorphisms $\eta$ of $M$ onto $M/\sim$.  If $\eta$ is sufficiently
close to $\kappa$ they are homotopic (since $M/ \sim \,\in \ANR\ $).  Also,
$\kappa (A \times Q)$ is a Z-set in $M\times Q/\sim $.  Therefore we
may adjust $\eta$ to equal $\kappa$ on $A\times \{\bar 0\}$. Let $g:M\setminus A\times \{\bar 0\}\to M\setminus A\times \{\bar 0\}$ be the restriction of $\kappa^{-1}\eta$. Now the action $\beta(h,x) =g^{-1}\alpha (h,g(x))$ on $M\setminus A\times \{\bar 0\}$ is free and extends by the identity over
$A\times \{\bar 0\}$ to a semifree action on $M$.

If $A$ is the fixed point set of a nontrivial semifree $H$-\ANR\ action
then $A$ must be an absolute neighborhood extensor for
metric spaces with trivial actions, so $A$ is an \ANR\ and $A^Q$ is a $Q$-manifold.
Let $\sigma$ be a based-free $H$-action on $Q$ (Theorem \ref{point}) with fixed point $\bar 0$. The fixed point
set of the diagonal action $(\alpha,\sigma)$ of $H$ on $M\times Q$ is $A\times \{\bar 0\}$.
Since $\sigma$ and $\alpha$ are $H$-\ANR\ actions, the diagonal action $\bar \beta$
of $H$ on $M\times Q(H)$ is also an $H$-\ANR\  action with fixed point set $A\times
\{\bar 0\}$, which is a Z-set in $M\times Q(H)$.

For a homeomorphism $\theta:M\times Q(H)\to M$ that is arbitrarily close to $p$ (\cite [Theorem 15.1]{chapman:lectures}), set $A'=\theta(A\times\{\bar 0\}),\beta =
\theta\,\bar\beta\,\theta^{-1}$, and $\zeta = p\,\theta^{-1}$.

 For $A^Q$, replace $M \text{ by } M\times Q,\alpha \text{ by } \alpha \times id_Q,
 \sigma \text{ by } id_Q \times \sigma, \text{ and } \theta \text{ by } \theta \, \tau,
 \text{ where }
 \tau:Q\times Q(H)\approx Q(H).$
 \end{proof}

\subsection{Applying Cappell-Weinberger-Yan I: Cellular Actions}

Cappell, Weinberger, and Yan give algebraic K-theoretic obstructions to the existence of
semifree actions of $G$ in
the two categories of finite cell complexes and cellular maps and of compact \ANR's and
continuous maps.  Their constructions
produce equivariant \ANR's (see Section \ref{preliminaries}).

\begin{definition} We make the following definitions.
\begin{enumerate}

\item A {\em cell structure} on a compact $Q$-manifold $N$ is a homeomorphism $h:K^Q\to N$, where $K$ is a finite cell complex.

\item a map $f: N\to
M$ of compact $Q$-manifolds is
{\em cellular} provided that there are cell complexes $K$ and $L$, a cellular map $g:K\to L$, and
cell structures $h_1:K^Q\to N$ and $h_2:L^Q\to M$ such that $f=h_2g^Qh_1^{-1}$.

\item An action $\alpha:G\times M\to M$ of $G$ on a $Q$-manifold $M$
is {\em cellular} provided that it is the stabilization of an action $\beta$ of $G$
on a finite cell complex $K$ by cellular homeomorphisms,  i.e., there is a cell
structure $h:K^Q\to M$ on $M$
 such that for each $g\in G,\alpha (g,.) = h\beta(g,.)^Qh^{-1}$.
 \end{enumerate}
\end{definition}

Cellular actions are $G$-\ANR\ actions \cite{Antonyan01}.

We use $C_{\ast}(X)$ for the integral cellular or singular chain complex of $X$ as appropriate and $\tilde H_i$ for
reduced homology.  If $g_{\ast}:C_{\ast}\to D_{\ast}$ is a chain map of chain complexes then its
mapping cone $C_{\ast} ^{cone}(g_{\ast})$ is defined as $C^{cone}_n(g_{\ast})=C_{n-1}\oplus D_n$
with $\partial=(-\partial_C,\partial_D-g_{n-1})$ \cite{Weibel}.
The mapping cylinder of a map $f:A\to B$ of spaces is $M(f)=(A\times [0,1])\cup B/\sim$ where
$(x,1)\sim f(x)$.  We identify $A$ and $B$ with their images $A\times\{0\}$ and $B$ in
$M(f)$.
The mapping cone of $f$ is $C(f)=(A\times [-1,1])\cup B/\equiv$, where $(x,-1)\equiv (x',-1)$ for all
$x,x'\in A$ and $(x,1)\equiv f(x)$.  The algebraic mapping cone $C_{\ast}^{cone}(f)$ of $f$ is $C_{\ast}^{cone}
(f_{\ast})$, where $f_{\ast}:C_{\ast}(A) \to C_{\ast}(B)$.

In the following lemma, we record several more facts for reference.
Suppose that $f_i:A_i\to B_i$ are cellular maps of connected finite
cell complexes, $i=1,2$, and that $g:A_1\to A_2$ and $h:B_1\to B_2$ are cellular homotopy
equivalences with $f_2g$ homotopic ($\simeq$) to $hf_1$.
Set $\pi=\pi_1(B_2)$ and identify it with the group of deck transformations of
the universal cover $\tilde B_2$.  This gives $\pi$-actions on induced covering spaces and mapping cylinders
and $\mathbb{Z}[\pi]$-chain complex structures on the associated cellular chain complexes and algebraic mapping cones.

\begin{lemma}\label{cylinderlemma}${}$

\begin{enumerate}

\item  $ (M(f_1),A_1)$ and $(M(f_2),A_2)$ are homotopy equivalent as pairs, and with all maps and homotopies preserving $M(f_1)\setminus A_1$ and $M(f_2)\setminus A_2$ as well as $A_1$ and $A_2$,

\item $C_{\ast}(\tilde f_1)$ is $\mathbb{Z}[\pi]$-chain homotopy equivalent to $C_{\ast}(\tilde f_2)$,

\item if $f_2g=hf_1$, and $\psi: (M(f_1),A_1)\to (M(f_2),A_2)$ is the induced map, and if $g$ and $h$ are simple homotopy equivalences, then
the torsion of the
(algebraic) mapping cone $C^{cone}_{\ast}(\tilde \psi^{cone}_{\ast})$ of the induced chain homotopy equivalence
${\tilde \psi_{\ast}^{cone}}:C_{\ast}^{cone}(\tilde f_1)\to C_{\ast}^{cone}(\tilde f_2)$ is zero in Wh($\pi$).

\end{enumerate}
\end{lemma}

\begin{proof} Let $\Lambda:A_1\times [0,1]\to B_2$ be a homotopy from
$f_2g=\lambda_0 $ to $hf_1=\lambda_1$, and let $\zeta:(M(f_1),A_1)\to (M(f_2),A_2)$ be the map $(a,t)
\mapsto (g(a),2t)$, for $t\le \frac{1}{2}$, $(a,t)\mapsto \lambda_{2t-1}(a)$ for
$\frac{1}{2}\le t\le 1$, and $b\mapsto h(b)$, for $b\in B_1$.  Sliding along the $J$
coordinate of $M(\zeta) = (M(f_1)\times J) \cup M(f_2)/\sim, J = [0,1]$, gives a
deformation retraction of $M(\zeta)$ to $M(f_2)$ that restricts to a deformation retraction of
$M(g)$ to $A_2$. To verify that $\zeta$ is a homotopy equivalence of pairs it suffices to
observe that there is a deformation retraction of $M(\zeta)$ to $M(f_1)$ that restricts to
a deformation retraction of $M(g)$ to $A_1$.  Since $g$ and $h$ are homotopy equivalences, there are by Fox's Theorem
\cite[Corollary 0.21]{Hatcher01} deformation retractions $F^1:M(g)\times [0,1]\to M(g)$ of $M(g)$ to $A_1$
and $F^2:M(h)\times  [0,1]\to M(h)$ of $M(h)$ to $B_1$.  It is immediate that $F^1$ extends to a homotopy $F^3$
of $M(\zeta)$ that is stationary on $\{((a,t),s)\in M(\zeta):t\ge\frac{1}{2}\}\cup M(f_1)\cup M(h)$
that preserves the $t$-coordinate.
Then,
since each $n$-cell of $A_1$ generates an $n+2$-cell
of $M(\zeta)$ with a ``free face" in $M(g)$, there is a cellular collapse $F^4$ of $M(\zeta)$ to $M(f_1)\cup M(h)$. Applying $F^3,F^4$, and $F^2$ in succession produces the desired deformation retraction of pairs from $
(M(\zeta),M(g))$ to $(M(f_1),A_1)$ demonstrating~(1).  {By the Cellular Approximation Theorem, $F^1,F^2, F^3,$ and $F^4$ may be assumed to be cellular maps.}
This deformation  induces a $\mathbb Z[\pi]$-chain contraction of
$C_{\ast}(\widetilde{M(\zeta)}/C_{\ast}(\widetilde {M(f_1)}\cup \widetilde{M(g)})$ which immediately implies (2).

For (3), choose base points appropriately.  Following our notation convention,
 $\tilde B_1=h^{\ast}(\tilde B_2)$, $\tilde A_2 = f_2^{\ast}(\tilde B_2)$, $\tilde A_1=\tilde f_1^{\ast}h^{\ast}(\tilde B_2)=g^{\ast}
(\tilde A_2)$, $\tilde g=\tilde A_1\to \tilde A_2$, and the groups of deck transformations of $
\tilde B_i$ and $\tilde A_i$ are identified with $\pi$.   Let $\hat A_i$ be the universal cover of $A_i$.  We
identify the fundamental groups of $A_1$ and $A_2$ by the isomorphism $g_{\ast}$ induced by $g$ and
denote it by $\hat \pi$.   Now identify  $\hat A_1$ with the pullback $g^{\ast}(\hat A_2)$ and set $\hat g =
g_{\ast}: \hat A_1\to \hat A_2$.  The groups of deck transformations of $\hat A_i$ are identified with $
\hat\pi$ using $\hat g$.

We first reduce to showing that $\tau(C^{cone}_{\ast}(\tilde g))=0$.  Note that $C^{cone}_{\ast}
(\tilde\psi_{\ast}^{cone})= C^{cone}_{\ast}(\tilde\psi)/C^{cone}_{\ast}(\tilde g)$
so there is a short exact sequence of $\mathbb{Z}[\pi]$-chain complexes $0\to C^{cone}_{\ast}(\tilde g)\to C^{cone}
_{\ast}(\tilde \psi)\to
C^{cone}_{\ast}(\tilde\psi_{\ast}^{cone})\to 0$.  Using $\pi$-lifts of cells in  $M(\psi)$ to give preferred
bases, the Whitehead torsion of the chain complexes satisfies
$\tau(C^{cone}_{\ast}(\tilde\psi))=
\tau(C^{cone}_{\ast}(\tilde g))+\tau(C^{cone}_{\ast}(\tilde\psi_{\ast})^{cone}))$ in Wh($\pi$) \cite [17.2]
{cohen}.
As $h$ is a simple homotopy equivalence, so is
 $h_1: M(f_1)\to  M(f_1)\cup M(h)$.  The collapse analogous to $F^4$ demonstrates that the
 inclusion $h_2:M(h)\cup M(f_1)\to
 M(\psi)$ is a simple homotopy equivalence.  Therefore the inclusion $h_2h_1:M(f_1)\to M(\psi)$
 and hence $\psi$ are simple homotopy equivalences.  So $\tau(C^{cone}_{\ast}(\tilde\psi))=0$.

{We now verify that
$\tau(C^{cone}_{\ast}(\tilde g))=0$.  Set $\hat \pi = \pi_1(A_2)$ and identify it with $\pi_1(A_1)$ using the
isomorphism $g_{\ast}:\pi_1(A_1) \to \pi_1(A_2)$.  Let $\hat A_i$ be the universal cover of $A_i, i= 1,2$, and
denote by $\hat g:\hat A_1 \to \hat A_2$ the isomorphism of covering spaces induced by $g$.  Let
$\check A_i\subseteq \tilde A_i$ be the component containing the base point and $\check g:\check A_1 \to \check A_2$ be the restriction of $\tilde g$.  Denote by $p_i:\check A_i \to A_i$ the restriction to
$\check A_i$ of the projection of $\tilde A_i$ to $A_i$, and let
$\check\pi\subset \pi$ be the subgroup of $\pi$ that acts on $\check A_i$ as deck transformations of the
covering projection $p_i$.
Set $\bar \pi = \pi_1(\check A_1)$ and identify it with its image $\bar \pi$ in $\pi_1(\check A_2)$.
Using $p_{i\ast}$, also identify $\bar\pi$ with its images in $\hat\pi = \pi_1(A_i)$ and thus with subgroups of the
group $\hat\pi$ of deck transformations
of $\hat A_1$ and $\hat A_2$.  The covering projections $\hat A_i/\bar\pi \to A_i$ are equivalent to $p_i:\check
A_i \to A_i$ as both
fundamental groups project to $\bar\pi\subseteq \hat\pi$, so
the choice of base points determines covering space isomorphisms $k_i: \hat
A_i/\bar \pi \to \check A_i$ identifying the group
of deck transformations of $\hat A_i/\bar\pi \to A_i$ with those $\check \pi$ of $p_i:\check A_i \to A_i$, i = 1,2}.

Because $\check A_i$ is a connected
component of $\tilde A_i=f_i^{\ast}(\tilde B_i)$ and  $\tilde B_i$ is simply connected, $\bar \pi=ker
(f_{i\ast}:\hat\pi=\pi_1(A_i)\to \pi_1(B_i)=\pi)$ and is normal in $\hat \pi$.  Thus, the  group of deck
transformations of $\hat A_i/\bar \pi =\hat \pi/\bar \pi$ and therefore
$\check \pi = \hat \pi /\bar\pi$.  Let $\phi:\hat\pi\to \check\pi$ be the 
homomorphism $\hat \pi \to \hat \pi/\bar \pi \overset{k_{i\ast}}\longrightarrow \check \pi$ and
$j_i:\check \pi \to \pi$ be the inclusion.  (By our identifications these are independent of $i$.)

Since $A_i$ is connected,  $\tilde A_i\cong\pi\times_{\check\pi}\check A_i$, where $\pi\times_{\check\pi}\check A_i=\pi\times
\check A_i/\sim$, with $(\zeta,a)\sim (\zeta\eta,\eta^{-1}{(a)})$ for $\zeta\in \pi, a\in \check A_i$
and
 $\eta\in
\check\pi$.  The maps $\check g$ and $\hat g$ covering $g$ have mapping cylinders $M(\check g)$ and
$M(\hat g)$ which are covering spaces of $M(g)$.  Thus, $C^{cone}_{\ast}(\check g)\cong \mathbb{Z}[\check\pi]
\otimes_{\phi_{\ast}}C^{cone}_{\ast}(\hat g)$ and $C^{cone}_{\ast}(\tilde g)\cong \mathbb{Z}[\pi]\otimes_{j_{\ast}}C^{cone}
_{\ast}(\check g)$.  So $C^{cone}_{\ast}(\hat g)$,
$C^{cone}_{\ast}(\check g)$, and $C^{cone}_{\ast}(\tilde g)$ are free $\mathbb{Z}[\hat \pi]$-, $\mathbb{Z}[\check\pi]$-, and $\mathbb{Z}[\pi]$-chain complexes, respectively, with bases
lifts {$ \hat\sigma $, $\check \sigma = 1\otimes_{\phi_{\ast}}\hat \sigma$, and $\tilde\sigma =
1\otimes _{j_{\ast}}\check \sigma $ of cells {$\sigma$ of $A_i$}. }

Therefore, the change of rings homomorphisms satisfy $\tau(C^{cone}_{\ast}(\check g))=\linebreak
\phi_{\ast}(\tau(C^{cone}_{\ast}(\hat g)))\in Wh(\check \pi)$ and
$\tau(C^{cone}_{\ast}(\tilde g))=j_{\ast}(\tau(C^{cone}_{\ast}(\check g)))\in Wh(\pi)$.  As $g$ is a simple
homotopy equivalence, $\tau(C^{cone}_{\ast}(\hat g))=0$, so $\tau(C^{cone}_{\ast}(\tilde g))=0\in Wh(\pi)$.
\end{proof}

The following is a stabilized version of \cite[Theorem 1]{cwy}.  Here $\pi$ is identified with the covering
transformations of $\tilde M$ and $H_{\ast}(M;\mathbb F_p\pi)$ denotes homology with local coefficients, which may
be computed using the $\mathbb F_p$-chains of
$\tilde M$ with multiplication by elements of $\pi$ given by covering transformations \cite[Section 3.H]{Hatcher01}.  From this perspective,
$H_i(A;\mathbb F_p\pi)$ is computed from the $\mathbb F_p$-chains of
 the restriction $\tilde A$ to $A$ of $\tilde M$.  $\tilde H_{\ast}(A;\mathbb F_p\pi)$  is
reduced  homology.
  Our notation differs slightly from that
of \cite{cwy} in that $C_{\ast}^{cone}(\tilde f)$ is their $C_{\ast}(\tilde f)$, and our $\bar f$ corresponds to their $g$.

\begin{theorem}\label{QM:cellhomotopy} Let $A\subseteq M$ be a compact
$Q$-submanifold of $M$ and
denote the inclusion $A\to M$ by $f$.
Then the following are equivalent, where $\pi$ is the fundamental group of $M$.
\begin{enumerate}

\item There exist a compact $Q$-manifold $N$ containing $A$,
a semifree cellular action of
$G$ on $N$
with fixed point set $A$, and an invariant map that is an (inequivariant) homotopy equivalence $\bar f:N\to M$
extending $f$.

\item \begin{enumerate}

\item $f$ induces isomorphisms from $\tilde H_i( A;\mathbb{F}_p\pi)$ to $\tilde
H_i( M ;\mathbb{F}_p\pi)$ for all $i$ and for each prime $p$ dividing the order of $G$,
and

\item an obstruction $o(f) \in \tilde K_0 (\mathbb Z[\pi \times G])$ vanishes.\end{enumerate}
\end{enumerate}
\end{theorem}

\begin{proof}  Assuming that 2(a) is satisfied,
define $o(f)$ as follows.
Let $\mu:L^Q\to A$ and $\nu:S^Q\to M$
be triangulating homeomorphisms, and define
$\gamma:L \to L\times \{\bar 0\}
\subseteq L^Q$, $\delta:S\to S^Q$, $\epsilon:L^Q\to L$, and $\zeta:S^Q\to S$ be the inclusions and projections.  Set
$h=\zeta\nu^{-1}f\mu\gamma:L\to S$.
Choose a cellular approximation $\phi$ to $h$, and let $\hat L =\phi^{\ast}(\tilde S)$.  Then
let $\tilde \phi = \phi_{\ast}:\hat L\to \tilde S$ be the
canonical map and denote by $C^{cone}_{\ast}(\tilde \phi)$ the algebraic mapping cone of $\tilde \phi$.
A choice of lift $\tilde e$ for each cell $e$ of $L$ and $S$ gives a basis for $C^{cone}_{\ast}(\tilde \phi)$ as a finitely
generated free $\mathbb{Z}[\pi]$-chain complex and as a finitely generated $\mathbb{Z}[\pi\times G]$-chain complex where $G$
acts trivially on the modules.

In \cite{cwy}, $C^{cone}_{\ast}(\tilde \phi)$ is shown to determine an element $[C^{cone}_{\ast}(\tilde\phi)]$ of
$\tilde K_0(\mathbb{Z}[\pi\times G])$.  This is the obstruction in \cite[Theorem 1]{cwy}.  Set $o(f) = [C^{cone}_{\ast}(\tilde \phi)]$.  This element is an invariant of
the $\mathbb{Z}[\pi\times G]$-chain homotopy equivalence
class of $C^{cone}_{\ast}(\tilde \phi)$. If $\phi_1$ is a second cellular approximation of $h$, it is homotopic to $\phi$,
hence cellularly homotopic to $\phi$ by the relative Cellular Approximation Theorem.
If $\Phi$ is a cellular homotopy from $\phi=\phi_0$ to $\phi_1$, then its mapping cylinder
$M_{\Phi}$ cellularly deformation retracts to
$M_{\phi}$ and to $M_{\phi_1}$ relative to $S$ extending deformation retractions of $L\times [0,1]$ to
$L\times \{0\}$ and to $L\times \{1\}$.  Setting $\check L= \tilde\Phi^{\ast}(\tilde S)$ and
$\check\Phi = \tilde \Phi_{\ast}:\check L\to \tilde S$, we see that
these deformations lift to $\pi$-equivariant deformations of $M_{\check \Phi}$ to $M_{\tilde\phi}$ and to
$M_{\tilde\phi_1}$ relative to $\tilde S$ that keep $\check L$ in itself.  The deformations produce $\mathbb{Z}[\pi]$-chain homotopies showing that $C^{cone}_{\ast}
(\phi)$ and $C^{cone}_{\ast}(\phi_1)$ are $\mathbb{Z}[\pi]$-chain homotopy equivialent $\mathbb{Z}[\pi]$-chain complexes,
and as $G$ acts trivially, they are chain homotopy equivalent $\mathbb{Z}[\pi\times G]$-chain complexes.
Therefore $[C^{cone}_{\ast}(\tilde\phi)]\in \tilde K_0(\mathbb{Z}[\pi\times G])$ is independent of the
choice of cellular approximation $\phi$ of $h$.

To see that $o(f)$ is independent of the choice of triangulations $\mu$ and $\nu$, let
$\mu_i:L_i^Q\to A$ and
$\nu_i:S_i^Q\to M$ be triangulations, $i=1,2$, with associated injections and projections
$\gamma_i:L_i\to L_i^Q, \delta_i:S_i\to S_i^Q, \epsilon_i:L_i^Q\to L_i$, and $\zeta_i:S_i^Q\to S_i$.
Define $h_i=\zeta_i\nu_i^{-1}f\mu_i\gamma_i:L_i\to S_i, \eta =
\epsilon_2\mu_2^{-1}\mu_1\gamma_1:L_1\to L_2$, and $\theta=
\zeta_2\nu_2^{-1}\nu_1\delta_1:S_1\to S_2$.  These are not necessarily cellular maps.
However, we have shown that our obstruction is
independent of the choice of cellular approximation to the $h_i$'s, so we shall for simplicity of
notation henceforth assume that $h_i$ has been replaced by a cellular approximation and write
$C^{cone}_{\ast}(\tilde h_i)$ instead of $C^{cone}_{\ast}(\tilde \phi_i)$.

Our notation convention on pull-backs specifies covering spaces $\tilde A, \widetilde {L^Q_i}, \tilde L_i, \widetilde{S^Q_i}$,
and $\tilde S_i$ and maps $\tilde f, \tilde \mu_1, \tilde \gamma_i$ and $ \tilde\nu_i$ pulled back from the universal cover, $\tilde M$ of $M$.  It also
specifies
$\tilde \epsilon_i=\epsilon_{i\ast}:\epsilon^{\ast}(\tilde L_i)\to \tilde L_i$ and
$\tilde \zeta_i=\zeta_{i\ast}:\zeta^{\ast}(\tilde S_i)\to \tilde S_i$.  Now, $\epsilon_i^{\ast}(\tilde L_i)$ and $\zeta_i^{\ast}(\tilde S_i)$ are not $\widetilde {L_i^Q}$ and $\widetilde {S_i^Q}$: since $L_i^Q= L_i\times Q$ and
$S_i^Q=S_i\times Q$ and $\epsilon_i$ and $\zeta_i$ are the projections, $\epsilon_i^{\ast}(\tilde L_i)$
and $\zeta_i^{\ast}(\tilde S_i)$ are
$\tilde L_i\times Q$ and $\tilde S_i\times Q$, respectively, and $\tilde \epsilon_i$ and $\tilde \zeta_i$
are the projections.  There are natural covering space isomorphisms $\xi_i:\widetilde {L^Q_i}\to
\epsilon_i^{\ast}(\tilde L_i)$ and $\kappa_i:\widetilde {S^Q_i}\to \zeta_i^{\ast}(\tilde S_i)$ generated by the
Covering Homotopy Property by lifting to $\widetilde {L_i^Q}$ and $\widetilde {S_i^Q}$ the straight line
homotopies of $L_i^Q$ and  $S_i^Q$ given by $((x,y),t)\mapsto (x,(1-t)y)$ with initial lifts the identities of
$\widetilde{L_i^Q}$ and $\widetilde {S_i^Q}$.

We now have maps $\tilde h_i=\tilde \zeta_i\kappa_i\tilde \nu_i^{-1}\tilde f \tilde \mu_i\tilde
\gamma_i:
\tilde L_i\to \tilde S_i,\tilde \eta=\tilde\epsilon_2\xi_2\tilde
\mu_2^{-1}\tilde\mu_1\tilde\gamma_1:\tilde
L_1\to \tilde L_2$, and $\tilde\theta=\tilde
\zeta_2\kappa_2\tilde\nu_2^{-1}\tilde\nu_1\tilde\delta_1:\tilde S_1\to
\tilde S_2$.  A diagram chase shows that $\eta$ and $\theta$ are homotopy equivalences and
that $\theta h_1\simeq h_2\eta$.  We need to show that $C^{cone}_{\ast}(\tilde h_1)$ is $\mathbb{Z}[\pi]$-chain
homotopy equivalent to $C^{cone}_{\ast}(\tilde h_2)$.  This follows from  Lemma \ref{cylinderlemma}(2)
and cellular approximation.  This completes the argument that $o(f)$ is well-defined.

If there is a compact $Q$-manifold $N$ containing $A$ and a cellular semifree action
$\alpha$ of G on $N$ with fixed point set $A$ and an invariant homotopy equivalence $\bar f:N\to M$
extending $f$, \cite[Section 2]{cwy} shows that condition 2(a) is satisfied.  The cellularity condition is that
there is a pair $(K,L)$ of cell complexes, a semifree cellular $G$-action $\beta$ on $K$
with fixed point set $L$, and a triangulation $ \mu:(K^Q,L^Q)\to (N,A)$ conjugating $\beta^Q$ to
$\alpha$.  Let $\delta:S\to S\times \{\bar 0\}\subseteq S^Q$ triangulate $M$ and $\zeta:S^Q\to S$ be the
projection.  Then
 $g = \zeta\nu^{-1} f \mu\gamma$  and $K, \beta$, and $ \bar g = K \to K\times {\bar 0}\to K^Q\overset{\mu}{\to} N\overset {\bar f} {\to} M\overset{\nu^{-1}}{\to}S^Q\overset {\zeta}{\to} S$ satisfy condition (1) of
\cite[Theorem 1]{cwy}.  Therefore, the obstruction $o(f) $ vanishes in
$\tilde K_0(\mathbb{Z}[\pi \times G])$, and condition (2) of the Theorem is satisfied.

On the other hand, if condition (2) holds,
then $o(f)=[C^{cone}_{\ast}(\tilde h)]=0$, so Theorem 1 of \cite{cwy} provides a cell complex $K$ containing $L$ as a
subcomplex, a semifree cellular action $\beta$ of $G$ on $K$ with fixed point set $L$, and a homotopy
equivalence $g:K\to S$ extending $h$.  Set $N= K^Q\cup_{\mu}A$, where $\mu$ identifies $L^Q$ with $A$, and let $\psi:K^Q\to N$ be the natural map,
which is a homeomorphism.  It induces a cellular
action $\psi\beta^Q \psi^{-1}$ on $N$.  Now define $f_0$ to be the homotopy equivalence
$\nu g^Q\psi^{-1}:N\to M$.  On $A$, $f_0=\nu g^Q \mu^{-1}=\nu h^Q\mu^{-1}\simeq
\nu \delta h \epsilon \mu^{-1}=\nu \delta\zeta \nu^{-1}f\mu \gamma\epsilon \mu^{-1}\simeq f$.  Applying the Homotopy Extension Theorem to $(N\times \{0\})\cup (A\times [0,1])$ yields a homotopy $F:N\times [0,1]\to M$ from $f_0$ to a map $f_1$ that extends $f$.  This is the desired
homotopy equivalence $\bar f$ of statement (1) of the Theorem.
  \end{proof}

{\begin{remark} Our applications are to subsets $A$ of $M$, so that is how we have formulated Theorem 6.4.  The statement and proof in \cite{cwy} are for a cellular map $f:L \to K$ of finite cell complexes.  The argument above extends to the case $f:A\to M$ where $A$ is a compact $Q$-manifold not contained in $M$ by approximating $f$ with a homotopic embedding.
\end{remark}}

\subsection{ Applying Cappell-Weinberger-Yan II: More Cellular Actions, \ANR\ Actions}

We first note that in general for compact
groups~$H$ that act based-freely on $Q$, the following holds.

\begin{theorem}\label{fixedZ-set} Let $X$ be a Z-set in $M$ and denote the inclusion $X\to M$ by~$f$.
If there exist a compact $Q$-manifold $N$ containing $X$, a semifree action $\alpha$ of
$H$ on $N$
with fixed point set $X$, and a homotopy equivalence $\bar f:N\to M$ extending $f$, then~$X$ is the fixed
point set of a semifree $H$-action on $M$ induced by a homeomorphism $g:N\to M$  that is homotopic to $
\bar f$  if and only if the Whitehead torsion $\tau (\bar f)\in Wh(\pi)$ is $0$, where $\pi$ is the fundamental group of $M$.  The induced action is an \ANR\
action if $\alpha $ is an \ANR\ action.
 \end{theorem}

\begin{proof} If there is such a $g$, then $\tau(g) = 0$ by \cite{chapman:simpleforANR's}, and so is $\tau(\bar f)$ as
$\bar f$ is homotopic to $g$.  On the other hand, if $\tau(\bar f) = 0$, by Theorem \ref{point} there is a based-free action~$\theta$ of $H$ on $Q$ with fixed point $\bar 0$.  Then
the diagonal action $\beta = \alpha \times \theta $ on $N^Q$ is a
semifree $H$-action on $N^Q$.  Its fixed point set $X\times \{\bar 0\}$ is a Z-set in $N^Q$,
and if $\alpha$ and $\theta$ are \ANR\ actions, so is~$\beta$.  Let $i: X\to N^Q$ be the inclusion $x\mapsto
(x,\bar 0)$.  The projection $N^Q \to N$ is homotopic to a homeomorphism $\zeta$
\cite[15.1]{chapman:lectures}.  Denote by $\gamma$ the $H$-action on $N$ induced from
$\beta$ by $\zeta$.  Since $\tau(\bar f)=0$, $\bar f$ is homotopic to a homeomorphism $h:N\to M$.  Now
$h\zeta i$ is an embedding of $X$ in $M$ as a Z-set and is homotopic to $f$, so by the Homeomorphism Extension Theorem (see Section \ref{infdim}), there is a
homeomorphism $\eta:M\to M$ homotopic to the identity that extends $(h\zeta i)^{-1}$.   Then $g=\eta h$
induces the desired action of $G$ on $M$ from $\gamma$.
\end{proof}

Stabilizing Theorem 3 of \cite{cwy} and applying Theorem \ref{fixedZ-set} yields the following result.  Theorem 3 of \cite{cwy} differs from Theorem 1 in that $g$ (our $\bar f$ ) is required to be a {\em simple} homotopy equivalence and the obstruction lies in $Wh^T(\pi\subset \pi \times G)$.
(See \cite{av} for the development of  $Wh^T$, where it is denoted by $Wh^T_1$.
We follow \cite{cwy} in denoting $Wh_0(\pi)$ by the more familiar $\tilde K_0(\mathbb{Z}[\pi])$ and also drop \cite{av}'s subscript on the Whitehead group $Wh_1(\pi)$ for consistency with our usage elsewhere.) We denote the orbit space by $M/\gamma$ and the orbit map by $q:M\to M/\gamma$, and for simplicity we abuse notation by identifying the fixed point set $A_1$ of $\gamma$ with its image in $M/\gamma$.

\begin{theorem}\label{QM:simple}  Let $A\subseteq M$ be $Q$-submanifold Z-set.
 Denote the inclusion by $f:A\to M$, and let $\pi$ be the fundamental group of $M$.  There is a
semifree cellular action $\gamma$ of $G$ on $M$ with fixed point set $A$ if and only if

\begin{enumerate}

\item $ f$ induces isomorphisms $\tilde H_i( A;\mathbb{F}_p\pi)\to \tilde H_i( M;\mathbb{F}_p\pi)$ for all $i$ and for each prime $p$ that divides the order of $G$,

\item an obstruction $\sigma( f)\in Wh^T(\pi \subset \pi \times G)$ vanishes.

\end{enumerate}
\end{theorem}

\begin{proof}
If $\gamma$ is a cellular semifree $G$-action on $M$ with fixed point set $A
$ as in the statement of the Theorem, then there is semifree cellular $G$-action $\alpha$ on a finite
cell complex $S$ with fixed point set $L$ with $h:L\to S$ the inclusion and a homeomorphism
$\nu:(S^Q,L^Q) \approx (M,A)$ such that $f = h^Q$ and $\gamma = \alpha^Q$.  Then
\cite[Theorem 3]{cwy} shows that
the class $ [C_{\ast}^{cone}(\tilde h)] = 0 \in Wh^T(\pi\subset\pi\times G)$.
Now $[C_{\ast}^{cone}(\tilde h)] $ maps to $o(f)\in \tilde K_0(\mathbb{Z}[\pi\times G])$ in the following exact sequence from \cite{av}:
$$
    \dots\to Wh(\pi)\overset{\beta} \to Wh^T(\pi\subset\pi\times G)\overset{\alpha}\to \tilde K_0(\mathbb{Z}[\pi\times G])\to \tilde K_0(\mathbb{Z}
    [\pi]).
$$

Define $\sigma(f) = [C_{\ast}^{cone}(\tilde h)] $.  The proof that
$\sigma(f)$ is well-defined is similar to the proof that $o(f)$ in Theorem \ref{QM:cellhomotopy} is well-defined.   We use the notation in its proof.  If $\nu_i:(S_i^Q,L_i^Q)\to (M,A),i=1,2$ are triangulations as pairs,
set $\mu_i=\nu_i|_{L_i^Q}$.  Also, let $\gamma_i:L_i\to
L_i\times \{\bar 0 \}\subseteq L_i^Q$ and $\delta_i:S_i\to
S_i\times \{\bar 0\}\subseteq S_i^Q$ be the indicated inclusions, and
let $\epsilon_i:L_i^Q\to L_i$ and $\zeta_i:S_i^Q\to S_i$ be the projections.  Then
$h_i=\zeta_i\nu_i^{-1}f\mu_i \gamma_i:L_i\to S_i$ are the inclusions.
Let $\eta:L_1\to L_2$ be $\epsilon_2\mu_2^{-1}\mu_1 \gamma_1$ and
$\theta = \zeta_2\nu_2^{-1}\nu_1 \delta_1$.  As $\eta^Q$ and $\theta^Q$ are homotopic to
the homeomorphisms $\mu_2^{-1}\mu_1$ and $\nu_2^{-1}\nu_1$, respectively, $\eta$ and $\theta$
are simple homotopy equivalences.  Because $h_1$ and $h_2$ are inclusions, $\theta h_1=h_2 \eta$ so
$\theta$ and $\eta$ define a map $\psi:M(h_1)\to M(h_2)$.  Then $\tilde \psi_{\ast}:C_{\ast}(M(\tilde h_1))
\to C_{\ast}(M(\tilde h_2))$ induces a chain map $\tilde \xi_{\ast}:C^{cone}_{\ast}(\tilde h_1)\to C^{cone}_{\ast}(\tilde h_2)$.  By Lemma \ref{cylinderlemma}(3), $\tau(C^{cone}_{\ast}(\tilde \xi_{\ast}))=0\in
Wh(\pi)$, so in $Wh^T(\pi\subset\pi\times G)$, $[C^{cone}_{\ast}(\tilde \xi_{\ast})]=\beta(\tau(C_{\ast}^{cone}(\tilde \xi_{\ast})))=0$.

Since $o(f)\in \tilde K_0(\mathbb{Z}[\pi\times G])$ is well-defined,
$\sigma(\tilde h_2)-\sigma(\tilde h_1 )\overset{\alpha}\to 0$.  Let  $(D,\partial_D)$ denote the chain complex with $D_n=C^{cone}_{n-1}(\tilde h_1)$ and $(\partial _D)_n= -\partial_{n-1}$, where $\partial_{n-1}: C^{cone}_{n-1}(\tilde h_1) \to C^{cone}_{n-2}(\tilde h_1)$ is the boundary map of $C^{cone}_{\ast}(\tilde h_1)$.
Then $[D_{\ast}]=-[C^{cone}_{\ast}(\tilde h_1)]\in Wh^T(\pi\subset\pi\times G)$.
Now $C^{cone}_{n}(\tilde\xi_{\ast})=D_{n}\oplus C^{cone}_{n}(\tilde h_2)$ and there is a short exact
sequence of chain complexes $0\to C^{cone}_{\ast}(\tilde h_2)\to C^{cone}_{\ast}(\tilde\xi_{\ast})\to D_{\ast}\to 0$.
By \cite[Lemma 1.8]{av},
$[C_{\ast}(\tilde\xi_{\ast})]=[D_{\ast}]+[C^{cone}_{\ast}(\tilde h_2)]=
[C^{cone}_{\ast}(\tilde h_2)]-[C^{cone}_{\ast}(\tilde h_1)]=\sigma(\tilde h_2)-\sigma(\tilde h_1)\in
Wh^T(\pi\subset\pi\times g)$.  Thus, $\sigma(\tilde h_2)-\sigma(\tilde h_1)=\beta(\tau(C^{cone}_{\ast}(\tilde \xi_{\ast})))=0$ and $\sigma(f)$ is independent of the triangulation used.

 If Conditions (1) and (2) are satisfied, then $\sigma( f) = [C_{\ast}^{cone}(\tilde h)]= 0$, so by  \cite[Theorem 3]{cwy} there are a
pair $(N,L)$ of finite cell complexes, a semifree cellular action $\alpha$ of $G$ on $N$ with fixed
point
set $L$, and an invariant simple homotopy equivalence $g:N\to S$ that extends $h$.
By \cite[Theorem 5.2 (implicit in the proof which is written for simplicial complexes)]{West70}, see \cite[Theorem 38.1]
{chapman:lectures}, $g^Q$ is homotopic to a homeomorphism $\hat g:N^Q\to S^Q$, and since
$\hat g|_{L^Q}\simeq h^Q$, we may require by the Homeomophism Extension Theorem (see Section \ref{infdim}) that in fact $\hat g|_{L^Q}=h^Q$.
Define $\gamma$ to be the cellular action on $M$ induced from $\alpha^Q$ by $\nu\hat g$.
\end{proof}

\begin{remark}${}$
\begin{enumerate}
\item For any other Z-embedding  $f_1:A\to A_1\subset M$ homotopic to $f$, triangulation $(T^Q,U^Q)\approx (M,A_1)$, and inclusion $\chi:U\to T$
there is by the Homeomorphism Extension Theorem a homeomorphism of pairs, $\kappa:(M,A)\to (M,A_1)$, so
$\lambda:S\to S\times {0}\to S^Q\overset {\kappa}{\to} T^Q\to T$ and $\mu= \lambda|_L:L\to U$ are simple homotopy equivalences.  By Lemma \ref{cylinderlemma}(3), the induced chain homotopy
equivalence $\tilde \lambda_{\ast}:C^{cone}_{\ast}(\tilde h)\to C^{cone}_{\ast}(\tilde\chi)$ between chain complexes of the cellular mapping
cones has vanishing torsion $\tau(\tilde \lambda_{\ast})$ and the argument for $\tilde\xi_{\ast}$ applies to
show that $\sigma(f_1)=\sigma(f)$, so $\sigma( f)$ depends only on the homotopy class of $f$ among the Z-embeddings of $A$.

\item Again, if $A$ is not a subset of $M$ and $f:A\to M$ is a map, then the
argument extends to this case by replacing $f$ with a homotopic embedding $f_1$ as a Z-set.
\end{enumerate}
\end{remark}

\begin{corollary}\label{QM:1conn} If $M$ is simply connected and $G = \mathbb{Z}_q$, for a prime
$q\le 19$, then a Z-set sub $Q$-manifold $A$ in $M$ is the fixed point set of a cellular $G$-action if and only if
the inclusion $f:A\to M$ induces
isomorphisms from $\tilde H_i( A;\mathbb{Z}_p)$ to $\tilde H_i (M ; \mathbb{Z}_p)$ for all~$i$ and for each
prime $p$ dividing the order of $G$.
\end{corollary}

\begin{proof}  In this case $\tilde M = M$.  In the exact sequence from \cite{av} displayed above,
$\pi = \{1\}$ as $M$ is simply connected.  Since $Wh(\{1\})$ and $\tilde
K_0([\mathbb{Z}])$ are $0$, we have
$\sigma( f) \in Wh^T(\pi\subset \pi \times G) \cong \tilde K_0(\mathbb{Z}[G])$ by the above exact sequence.  Since $\tilde K_0(\mathbb {\mathbb{Z}}[G])=0$ \cite[page 30]{milnor:algK}, we can apply Theorem \ref{QM:simple}.
\end{proof}

We apply \cite[Theorem 4]{cwy} in the following.  See \cite{cwy} for more information about~$\tilde
K^{top}_0$.

\begin{theorem}\label{QM:anrhomotopy}  Suppose $f : A \to M $ is inclusion, where $A$ is an
\ANR\ Z-set in~$M$.  Then $A$ is the fixed point set of a semifree $G$-\ANR\ action on a
compact $Q$-manifold~$N$
containing $A$ as a Z-set and there is an invariant homotopy equivalence $\bar f:N\to M$
extending $f$ if and only if the following are satisfied:

\begin{enumerate}

\item $f$ induces isomorphisms from $\tilde H_i( A;\mathbb{F}_p\pi)$ to $\tilde H_i( M;
\mathbb{F}_p\pi)$ for all $i$ and for each prime $p$ dividing the order of $G$, and

\item an obstruction $\omega(f) = [C^{cone}_{\ast}(\tilde f)] \in \tilde K^{top}_0 (\mathbb Z[\pi] \subset\mathbb Z[\pi \times G])$ vanishes, where $\pi$ is the fundamental group of $M$.
\end{enumerate}
If the Whitehead torsion $\tau(\bar f) \in Wh(\pi)$ is $0$, then $A$ is the fixed point set of a semifree
$G$-\ANR\ action of G on $M$ induced by a homeomorphism $g:N\to M$ homotopic to $\bar f$.
If $\tau(\bar f) \in f_{\ast}(Wh(\pi_1(A)))\subseteq Wh(\pi)$, then there is an \ANR\ Z-set $C\subseteq
M$ containing $A$ as a deformation retract that is the fixed point set of a semifree $G$-\ANR\ action
of G on $M$.
\end{theorem}

\begin{proof}  Since $Q$-manifolds are \ANR's the ``only if'' direction is covered by \cite[Theorem 4]
{cwy}.  If (1) and (2) hold, then Theorem \cite[Theorem 4]{cwy} gives an invariant homotopy
equivalence $\hat f:X\to M$ extending
$f$ for a compact \ANR\ $X$ containing $A$ as the fixed point set of a semifree $G$-\ANR\ action of
$G$.  Now set $N = X \times Q(G)$ (see Theorem \ref{standard action}) and let $i:X\to N$ by
$x\mapsto (x,\ast)$, where $\ast$ is the unique fixed point of $Q(G)$.  The diagonal action on
$N$ is a
semifree $G$-\ANR\ action with fixed point set $i(A)$, and $\bar f = \hat fq$, where $q$ is the
projection, is
invariant.  If $\tau(\bar f) = 0$, $\bar f$ is a simple homotopy equivalence and
Theorem \ref{fixedZ-set} applies.

If $\tau(\bar f) \in f_{\ast}(Wh(\pi_1(A)))\subseteq Wh(\pi)$, let $\sigma \in Wh(\pi_1(A))$ be such that $f_{\ast}(\sigma)= -\tau(\bar f)\in Wh(\pi)$.  Using Chapman's extension of simple homotopy
theory to compact \ANR's \cite{chapman:simpleforANR's}, choose a homotopy equivalence $g:A\to
B$, with $B$ a compact \ANR,
and with $\tau(g)=\sigma$.  Let $Z$ be the mapping cylinder of $g$.  Fix a pair  $(K,L)$ of cell
complexes with
$(K^Q,L^Q)$ homeomorphic to $(Z^Q,A^Q)$, and let $j:L\to K$ be the inclusion.  Then the $\mathbb
Z[\pi_1(A)]$-chain complexes of the $\pi_1(A)$-coverings $\tilde K$ and $ \tilde L$ induced
from the universal cover of $A$ with its $\pi_1(A)$-action satisfy
$[C_{\ast}(\tilde K)/C_{\ast}(\tilde L)]=[C^{cone}_{\ast}(\tilde j)]=\sigma \in Wh(\pi_1(A))$.

Let $\zeta:Z\to M$ be the
identity on $A$ and an embedding of $Z$ as a Z-set, and set $C=\zeta(Z)$.
Now let $N_C=N\cup C$.  Then $N_C$ is a $G$-\ANR\ containing $C$ as the fixed point set of the
semifree $G$-\ANR\ $G$-action extending the one on $N$, and the extension $\bar f_C$ of $\bar f$
to $N_C$ that is the identity on $C$ is a homotopy equivalence.  Thus
$$
    \tau (\bar f_C) = \tau(\bar f) + \bar f_{\ast}i_{\ast}(\tau(g)) = \tau(\bar f) + (\bar fi)_{\ast}
    (\sigma)=\tau(\bar f)
+ f_{\ast}(\sigma)=0\in Wh(\pi)
$$
and we are in the case of the first paragraph.  Apply Theorem \ref{fixedZ-set} to obtain the
desired $G$-action on $M$ from that on $N_C$.
\end{proof}

\begin{remark}\label{wh}${}$
\begin{enumerate}

\item Since $A$ has arbitrarily small open neighborhoods in $M$ of the form
$A\times Q\times [0,1)$, the fixed point set $C$ may be required to lie in any neighborhood of $A$.  If
$A$ is a $Q$-manifold, then $C$ may instead be required to be a deformation retract of $A$, using
the fact that the projection of $A\times Q$ to $A$ is homotopic to a homeomorphism \cite[15.1]
{chapman:lectures}.

\item If $\pi$ is $0$, $\mathbb{Z}^n$, $\mathbb{Z}_m$ for $m = 2, 3, 4, 6$ or a braid group,
then $Wh(\pi)=0$, \cite{bhs},\cite{milnor:wht},\cite{braids}, so $\tau(\bar f)=0$.
\end{enumerate}
\end{remark}

\section {Fixed-point sets that are Shape equivalent to \ANR's}

In this section, we study semifree $G$-actions on $Q$-manifolds with fixed point sets that are not necessarily
\ANR's but are Shape equivalent to compact \ANR's.  For an in-depth analysis of the latter question see
\cite{edgeo}.  We use an easily stated result of Geoghegan and Lacher in Corollary \ref{fd1-UV}.  In
Theorem \ref{anrShape}, we extend our Theorem~\ref {Shape} about actions on Keller spaces to $Q$-manifolds, and in Theorem \ref{shapeobstruction}
we adapt the K-theoretic obstruction of Theorem \ref{QM:anrhomotopy}.
These semifree actions are not $G$-\ANR\ actions if their fixed point sets are not
\ANR's, but they are related because their complements are the complements of $G$-\ANR\  actions.

In our context of $Q$-manifolds, it is
simplest to use Borsuk's original definition of Shape in \cite{borsuk:Shape}.  (It is now seen as very similar to pro-homotopy theory.)  For less elementary and more
sophisticated treatments, see  \cite{mardseg,dydseg, edhas, hol}.  These variants all deliver the
same isomorphism classification for compact (metric) spaces.  Borsuk
originally defined the Shape category as follows.  (We use bold face for Shape morphisms.)  Objects
are the compact subsets of the Hilbert cube.  A
\emph{fundamental sequence}) $\underline {f}$ from $B$ to $C$
is a sequence of maps $f_n:Q\to Q$ with the property that for every open neighborhood $V$
of $C$ there is a neighborhood $U$ of $B$ such that $f_n(U)$ is eventually
(i.e., for sufficiently large n) in $V$ and the $f_n$'s are
eventually homotopic in $V$ when restricted to $U$.  Two fundamental sequences
$\underline {f} =\{f_n\}_n$ and
$\underline {g} = \{g_n\}_n$ are homotopic ($\underline f \simeq \underline {g}$) provided that for every
open $V$ containing $C$,
there is a neighborhood $U$ of $B$ such that for sufficiently large $n$, $f_n|_U$ and
$g_n|_U$ are  homotopic in $V$.
The  \emph{Shape morphisms} $\boldsymbol f:B\to C$ are the homotopy classes of
fundamental sequences from $B$ to $C$ in this
sense.  Thus $\boldsymbol f$ is a Shape isomorphism provided that there is a morphism
$\boldsymbol g:C\to B$ such that representing fundamental sequences
$\underline f$ and $\underline g$ satisfy $\underline g\underline f \simeq \underline {\identity}_B$ and
$\underline f \underline g \simeq \underline {\identity}_C$.  The
spaces $B$ and $C$
have the same Shape if there is a Shape isomorphism from $B$ to~$C$.

It is immediate that this classification is
independent of the copy of the Hilbert cube and the embeddings of $B$ and $C$ in it, so we assume
that all compact spaces have been embedded in $Q$.  Note that that a map $f:B\to C$
generates a fundamental sequence $\{f_n\}_{n\ge 1}$ where each $f_n=\bar f$ where $\bar f$ is any
continuous extension of $f$ to $Q$, and that if $C$ is an \ANR, then
each fundamental sequence $\underline f:B\to C$ is homotopic to one generated by a map.
(To see this last point, let $ r:W\to C$ be a retraction to $C$ of
a closed neighborhood of $C$ in $Q$ and choose an extension $\bar r:Q\to Q$.  Now let
$V\subseteq W$ be a smaller closed neighborhood of $C$ that deformation retracts to $C$ in $W$
and let $\Lambda:V\times
[0,1] \to W$, be such a deformation retraction, with $\bar \Lambda :Q\times [0,1]\to Q$ any
extension of $\Lambda$.  Then $\underline  f \simeq \underline g$, where $g_n=\bar
\lambda_1f_n$.  If $n_0\ge 1$ is
such that there is a neighborhood $U$ of $B$ such that if $m,n\ge n_0$ then $f_m|_U$ and $f_n|_U$
are
homotopic in $V$, then $h=g_{n_0}|_B:B\to C$ has the property that for any extension $\bar h:Q\to
Q$ of
$h$, the fundamental sequence $\underline  h =\{h_n\}_{n\ge 1}$ where $h_n=\bar h$ for all $n$ is
homotopic
to $\underline g$,  hence to $\underline  f$.)

Consider a closed connected subset $X$ of $M$, i.e., a subcontinuum.  For $x_0\in X$ we use
$\check \pi_1(X,x_0) = \varprojlim \{i^{CD}_{\ast}:
\pi_1(C,x_0)\to\pi_1(D,x_0)\}$,
where
$C$ and $D$ vary over the compact \ANR\ neighborhoods of $X$ in $M$ and $x_0\in X$.  The
inclusion $f:(X,x_0)\to
(M,x_0)$ induces a homomorphism $f_\ast:\check \pi_1(X,x_0)\to \pi_1(M,x_0) = \pi$ which induces
homomorphisms $f_{\ast}:\tilde K_0(\mathbb Z [\check \pi_1(X,x_0)])\to \tilde K_0(\mathbb Z[\pi])$ and
$f_{\ast}:Wh(\check \pi_1(X,x_0))\to Wh(\pi)$.

A very quotable result of Geoghegan and Lacher \cite{geolach} which we cite in Corollary ~\ref{fd1-UV} is that a finite dimensional 1-UV compact
space is Shape equivalent to a finite cell complex if and only if its \v Cech cohomology with
integer coefficients is finitely generated.  A compact (metric) space $X$ satisfies the 1-UV condition if for some embedding of $X$ in $Q$
(and hence every embedding of $X$ in an \ANR), each neighborhood $U$ of~$X$ contains a
neighborhood $V$ of $X$ such that every map of $S^1$ into $V$ is homotopic to a constant map in
$U$.
For connected $X$, this is quickly implied, for example, by triviality of $\check \pi_1(X,x_0)$ together
with movability of $X$.  ($X$ is
\emph{movable} if for every neighborhood $U$ of $X$ in $Q$, hence in any \ANR, there is
a smaller neighborhood $V$ of
$X$ that can be deformed into any neighborhood of $X$ with the deformation occurring in
$U$.)

A space $Y$ {\em Shape dominates} $X$ provided that there are Shape morphisms $\mathbf u:X\to
Y$ and $\mathbf d:Y\to X$ such that $\mathbf d\mathbf u \,=  \,\mathbf {id}_X$.
D. A.\ Edwards and
Geoghegan have shown \cite{edgeo} that $X$ has the Shape of a compact \ANR\ (equivalently, a finite cell complex) if and only if $X$ is Shape dominated by an \ANR\ and the Shape version of the Wall finiteness obstruction in $\tilde K_0(\mathbb Z[\check \pi_1(X,x_0)])$
vanishes.

\begin{definition} Let $X$ be the fixed point set of a semifree action $\alpha$ of a compact group $H$
on a $Q$-manifold $Y$.  Call $\alpha$
\emph{quasitame} provided that $\alpha|_{Y\setminus X}$ is equivariantly
homeomorphic to $\beta|_{Y\setminus A}$ for some semifree $H$-action $\beta$ of $H$ on $Y$ with
\ANR\ fixed-point set $A$ that is a Z-set in $Y$.  Call $\alpha$ {\em semitame} if $\beta$ may be
taken to be an $H$-\ANR\ action.
\end{definition}

In view of the Theorem that if $X$ is a CE Z-set of $Q$ then the quotient
space $Q/X\approx Q$ \cite[Theorem 7.5.5]{vm:book}, Anderson's Conjecture is that if $X$ is a CE
Z-set in $Q$, then all involutions of $Q$ with fixed point set $X$ are semitame. Theorem~\ref{Shape}
extends to $Q$-manifolds as follows.

\begin{theorem}\label{anrShape} Let $A\subseteq M$ be an \ANR\,Z-set, $H$ be a compact
group and
$\alpha$ be a semifree action of $H$ on $M$ with fixed point set $A$.  Let $X$ be a Z-set in $M$
that is
Shape equivalent to $A$ and let $\boldsymbol i:X\to M$ and $\boldsymbol j:A\to M$ be the Shape
morphisms generated
by the inclusions $i$ and $j$, respectively.
If there is a Shape equivalence $\boldsymbol\kappa:X\to A$ such that $\boldsymbol j \boldsymbol
\kappa = \boldsymbol i$ then $X$ is the
fixed point set of a quasitame semifree action $\beta$ of $H$ on $M$.  If $\alpha$ is an
$H$-\ANR\ action, then $\beta$ may
be required to be semitame.
\end{theorem}

\begin{proof} We can arrange to appeal to the proof of Theorem \ref{Shape}, which takes
place in the Hilbert cube $K =C(A_1)\times Q$, where $A$ has been assumed without loss of
generality to
be $A_1\times Q$ and a compact group $G$ is acting semifreely on $K$.  In the proof of Theorem \ref{Shape}
a homeomorphism $h$ of $K$ is
constructed from
$\alpha$ and a fixed homeomorphism $f:K\setminus A\to K\setminus X$.
The conjugate $\beta = h^{-1}\alpha \,h$ of $\alpha$ is sufficiently well controlled that its conjugate
$\gamma =f\beta f^{-1}$ extends continuously to a semifree action with fixed point set $X$.

The homeomorphism $h$ is a slide $((x,t),u)\to ((x,\phi(t)),u)$ where
$\phi:\mathbb{I} \to \mathbb{I}$ is specified to send $c_i=2^{-i}$ to $d_{j(i)}$ for a decreasing sequence
$d_j$.
The condition on $d_1$ is only that it be less than $c_2$ and there is no condition on the restriction
of
$\phi$ to $[c_1,1]$ other than that $[c_1,1] \to [d_1,1]$.  By fixing $\epsilon > 0$ sufficiently small, redefining $c_i =
\epsilon \,2^{-i}$
for all $i$, and specifying that $\phi$ be the identity
on $[\epsilon,1]$, we may require that $h$ be the identity off any neighborhood of $A$ in $K$.

This allows us to apply the proof of Theorem \ref{Shape} in the current situation because we can
arrange for $M$ to
be embedded in $K$ as a closed neighborhood of $A$ and for $f$ to be the identity on $K\setminus
M$. If we extend $\alpha$ (discontinuously) to $K$ by the trivial action on $K\setminus M$, then $M$
is invariant under the action of $\gamma$.

To do this, we first note that we may assume that $A$ is a $Q$-manifold, since
there is a homeomorphism $\eta:M\times Q\to M$ homotopic to the projection \cite[Theorem 15.1]{chapman:lectures} and by the Estimated Homeomorphism Extension Theorem $\eta$ may be required to equal the projection on $A\times \{\bar 0\}$.  Then $\alpha$ may be replaced by the action induced from $\alpha \times \identity$ by $\eta$.  Similarly, as
$M\approx M\times \mathbb{I}$ we may write $M = M'\times \mathbb I$, where $M'$ is a copy of $M$ and $M' \times \{0\}$ contains the Z-set $A \cup X$.

Identifying $A$ with its copy $A_1\times Q = (A_1\times\{0\}) \times Q \subset C(A_1)\times Q = K$, we may further assume that $M'\times \{0\}$ is embedded
in $K$ as a Z-set containing $A_1 \times Q$.  Finally, since Z-set $Q$-submanifolds have closed collar
neighborhoods \cite[Theorem 16.2]{chapman:lectures}, we may assume that $M$ is embedded in
$K$ as a neighborhood of $M'\times \{0\}$ with boundary $M'\times \{1\}$.

It remains to construct the homeomorphism $f:K\setminus A \approx K\setminus X$ that is the
identity off $M$.
As $A$ is an \ANR, $\boldsymbol\kappa$ is represented by a map $\kappa:X\to A$.
Since $A$ is an \ANR, $\kappa$ extends to a neighborhood $V_1 \subseteq M'\times [0,1)$ of $X$.  Using
$K\approx Q$, let $\{\lambda_m:K\to K\}_m$ be a fundamental
sequence representing
$\boldsymbol \lambda = \boldsymbol \kappa^{-1}$.  By inductively taking subsequences, we may
assume that there are neighborhoods $M'\times [0,1) \supseteq U_1\supseteq
U_2\supseteq \cdots$ of $A$ and $M'\times [0,1)\supseteq V_1\supseteq V_2\supseteq \cdots$  of $X$ with $
\bigcap \,U_m = A$ and $\bigcap\,V_m = X$, such that

\begin{enumerate}

\item $\lambda_m(U_m)\subseteq V_m$,

\item $\lambda_m|_{U_{m+k}}\simeq \lambda_{m+k}|_{U_{m+k}}$ in $V_m$ for all $k>0$,

\item $\lambda_m(\kappa|_{V_m})\simeq i|_{V_m}$ in $V_{m-1}$ for $m>1$, and

\item $\kappa (\lambda_m|_{U_m})\simeq i|_{U_m}$ in $U_{m-1}$.

\end{enumerate}

Since the inclusion $i$ of $X$ into the \ANR\ $M$ satisfies $i\simeq j\kappa$, $\lambda_m|_{U_m} =
i(\lambda_m|_{U_m})\simeq j\kappa (\lambda_m|_{U_m})\simeq j$  for large $m$.  (We are
here extending $i$ and $j$ to denote the inclusions of $V_1$ and $U_1$ into $M$.)

An examination of the argument for Theorem 25.2 in \cite{chapman:lectures} shows that $f$ may be
constructed
from the $\lambda_m$'s and $\kappa$
by using the
Homeomorphism Extension Theorem (see Section \ref{infdim}) inductively in a ``back-and-forth" argument to
produce
homeomorphisms $\mu_j$ of $K$ and  subsequences $U_{m_j}$ and $V_{m_j}$ of the
neighborhoods of $A$
and $X$  satisfying

\begin{itemize}

\item $\mu_j(U_{m_j})\subseteq V_{m_j}$,

\item $\mu_j|_{U_{m_j}} \simeq \lambda_{m_j}|_{U_{m_j}}$ in $V_{m_{j-1}}$, and

\item for $j > 1, \mu_j|_{K\setminus U_{m_{j-1} }}= \mu_{j-1}|_{K\setminus U_{m_{j-1}}}$.

\end{itemize}
Now $f = \lim_{n\to \infty} \mu_n\mu_{n-1}\cdots \mu_1$.
Therefore, if $\mu_1|_{K\setminus M} = \identity$, then $f |_{K\setminus M}$ is the identity.  In \cite{chapman:lectures} the only
requirement for $\mu_1$ is that $\mu_1(A)$ be contained in one of the $V's$, which we have
identified as $V_{m_1}$ above.
Because the $U$'s, $V$'s, and homotopies are in $M$ and $\lambda_m|_A \simeq j$, we
see that on the Z-set
$Y = A \cup (M'\times \{1\})$ of $M'\times \mathbb{I}$, the embedding
$\nu:Y\to M'\times \mathbb{I} $
that is $\mu_1$ on $A$ and inclusion on $M'\times \{1\}$ is homotopic to the identity and thus
extends to a homeomorphism of $M $ that is the identity on $M'\times \{1\}$ and then to a
homeomorphism  $\mu_1'$ of
$K$ that is the identity on the complement of $M$. If we replace $\mu_1$ by $\mu_1'$ in the
construction of $f$ then $f$ is the identity on $K\setminus M$.
\end{proof}

\begin{corollary} Let $A\subseteq M$ be an \ANR, $H$ be a compact
group and
$\alpha$ be a semifree action of $H$ on $M$ with fixed point set $A$.   If $X$ is a compact space
with $\Sh(X) = \Sh(A)$ then $X$ embeds in $M$ as the fixed point set  of a semifree action of $H$ on $M$.
\end{corollary}

\begin{proof} By Proposition \ref{collection}(3), we may replace $A$ by $A^Q$ and assume that it is a Z-set in
$M$.  We may then realize a Shape isomorphism by an embedding $\kappa:X\to A^Q$ and apply Theorem
\ref{anrShape}.
\end{proof}

The next theorem is an application of Theorem \ref{QM:anrhomotopy} in which the
subset $X$ of
$M$ is no longer required to
be an \ANR\  and we need to use \v Cech-style homology  to
formulate the Smith
Theory condition in \cite{cwy}.  We use reduced homology.  Thus,

\begin{definition} For a compact
set $X$ in a $Q$-manifold $M$, set
$\bar H_i(X;\mathbb{F}_p\pi)=\varprojlim_{\mathcal U}\tilde H_{i}(U;\mathbb{F}_p\pi)$, where $\mathcal U$ denotes the set of open
neighborhoods $U$ of $X$ in $M$ partially ordered by the inclusions $i^{UV}:U\to V$.
\end {definition}

\begin{remark} It is easily seen, using the
Triangulation Theorem and the compact Q-manifold neighborhoods of $X$ in $M$ and by embedding $M$ in $Q$, that
\begin{enumerate}

\item if $X$ is an \ANR\, then $\bar H_{\ast}(X;\mathbb{F}_p\pi)$ is isomorphic
to $\tilde H_{\ast}(X;\mathbb F_p\pi)$,

\item $\bar H_i(X;\mathbb{F}_p\pi)$ is isomorphic to the result of applying a \v
Cech construction using nerves of covers,

\item $\bar H_i(X;\mathbb{F}_p\pi)$ is functorial on the Shape category (cf.\ \cite{dydseg}).
\end{enumerate}

\end{remark}

The following is one result that is presented as two Theorems for clarity.

\begin{theorem}\label{shapeobstruction} Let $G$ be a finite group and $X$ be a Z-set in
$M$
such that the inclusion $f:X\to M$
induces isomorphisms from $\bar H_i( X;\mathbb F_p\pi)$ to
$\tilde H_i( M ;\mathbb F_p\pi)$ for all $i$ and for each prime $p$ dividing the order of $G$.
If $X$ has the Shape of a compact \ANR\
then $f$ determines an element
 $\omega(f) \in \tilde K^{top}_0 (\mathbb Z[\pi] \subset \mathbb Z[\pi \times G])$, where $\pi$ is the
 fundamental group of $M$.  If
 $\boldsymbol{\kappa} :X\to A$ is a Shape equivalence between $X$ and a compact \ANR\  $A$,
 then there is
 an embedding $\eta$ of $A^Q$ in $M$ as a Z-set $A_1$ containing $X$ such that $\omega(f) =
 \omega(j)$, where $j:A_1 \to M$ is
 the inclusion and $\omega(j)$ is the obstruction of Theorem \ref{QM:anrhomotopy}.
 \end{theorem}

\begin{proof}  Let $\boldsymbol \kappa:X\to A$ be a Shape equivalence
between $X$ and a compact
\ANR\  $A$ and let
$\zeta:X\to A$ be a map inducing
$\boldsymbol \kappa$.  Using the inclusion $A\to A \times \{\bar 0\}\to A^Q$, we
may replace $A$ by the homotopy equivalent $Q$-manifold $A^Q$ and assume
that $\zeta$ is an embedding of $X$ in $ A^Q $ as a Z-set.  Extend
$\zeta^{-1}$ to an embedding $\eta$ of $A^Q$ in $M$ as a Z-set $A_1$ and let
$\kappa$ denote the inclusion $X\to A_1$, which induces a Shape equivalence.

As $\bar H_i(X;\mathbb{F}_p\pi)$ is a Shape invariant,
$\kappa_{\ast}:\bar H_i(X;\mathbb{F}_p\pi)
\cong \bar H_i(A_1;\mathbb{F}_p\pi)\cong \linebreak
\tilde H_i(A_1;\mathbb{F}_p\pi)$.  By functoriality,
$j_{\ast}:\tilde H_i(A_1;\mathbb{F}_p\pi)\cong \tilde H_i(M;\mathbb{F}_p\pi)$ for
all $i$ and all primes
$p$ dividing the order of $G$.  Let $\omega(f) = \omega(j)\in
\tilde K_0^{top}(\mathbb Z[\pi]\subset\mathbb Z[\pi \times G])$ be the
Cappell-Weinberger-Yan obstruction identified in Theorem
\ref{QM:anrhomotopy}.  Then $\omega(j)$ is the image of $o(j)$ under $\phi$ in the following exact sequence
$$
    H_0(\pi;\boldsymbol{\tilde K}(\mathbb{Z}[G])) \to \tilde K_0(\mathbb Z[\pi \times G])\overset{\phi}\to
    \tilde K_0^{top}(\mathbb{Z}[\pi]\subset \mathbb{Z}[\pi\times G])
$$
in the proof of \cite[Theorem 4]{cwy}.

The  \ANR\ $A$ is determined only up to homotopy type by $\boldsymbol {\kappa}$, and $\eta$ is
restricted only by the
requirement that the inclusion $j\kappa $ be homotopic to $f$.  The proof of \cite[Theorem 4]{cwy}
explains why different
choices of $A, \boldsymbol {\kappa}$, and
$\eta$ vary $o(j)$ by elements of the kernel of $\phi$, showing that
$\omega(f)$ is uniquely determined by $o(j)$.
\end{proof}

For the following results, we use the notation of Theorem \ref{shapeobstruction}.

\begin{theorem}\label{semitame obstruction}  Let $G$ be a finite group and $X$ be a Z-set in
$M$ with the Shape of a compact \ANR\ $A$
such that the inclusion $f:X\to M$
induces isomorphisms from $\bar H_i( X;\mathbb{F}_p\pi)$ to
$\tilde H_i( M ;\mathbb{F}_p\pi)$ for all $i$ and for each prime $p$ dividing the order of $G$.
Then $\omega(f) = 0$ if and only if
there is a compact $Q$-manifold $N$
containing $A_1$ as the fixed point set of an \ANR\ action $\alpha$ of G on $N$ and an invariant
homotopy equivalence
$\bar f:N\to M$ that is the identity on $A_1$.  If the Whitehead torsion $\tau(\bar f)\in
Wh(\pi)$ lies in $f_{\ast}(Wh(\check \pi_1(X,x_0)))$,  then $X$ is the fixed point set of a semitame
semifree action of G on $M$.
\end{theorem}

\begin{proof} The ``if and only if'' part of the Theorem follows from Theorem \ref{QM:anrhomotopy} since Theorem \ref{semitame obstruction} shows that $\omega(f)=\omega (j)$ is uniquely determined in
$\tilde K_0^{top}(\mathbb Z[\pi] \subset\mathbb Z[\pi \times G])$.
Since $f = j\kappa$, $f_{\ast}(\text{Wh} (\check \pi_1(X,x_0)))
\subset j_{\ast}(\text{Wh}(\pi_1(A,x_0)))$.  Hence if
$\tau (\bar f) \in \linebreak f_{\ast}(\text{Wh}(\check \pi_1(X,x_0))$, Theorem \ref{QM:anrhomotopy}
provides an \ANR\ Z-set $C$ of $M$ containing $A_1$ as a deformation retract and the
fixed point set of a semifree $G$-action on $M$.  Now the inclusion of $X$ into $C$ is a shape
equivalence, so Theorem  \ref{anrShape} provides a semitame semifree
action of $G$ on $M$ with fixed point set $X$.
\end{proof}

\begin{corollary}\label{semitame:notorsion}
Suppose that $\pi = 0$, $\mathbb{Z}^n$ or $
\mathbb{Z}_m$ for $m = 2, 3, 4, 6$ or a braid group and that $X$ is a Z-set in $M$ that is
Shape equivalent to a compact \ANR\ $A$ and the inclusion $f:X\to M$ induces isomorphisms from $\bar H_i(  X;\mathbb{F}_p\pi)$ to
$\tilde H_i( M ;\mathbb{F}_p\pi)$ for all $i$ and for each prime $p$ dividing the
order of $G$.  If $\omega(f) = 0$, then $X$ is the fixed point set of a semitame semifree
action of $G$ on $M$.
\end{corollary}

\begin{proof}
Theorem \ref{semitame obstruction} applies, since $Wh(\pi) = 0$.
\end{proof}

\begin{corollary}\label{fd1-UV} Let $X$ be a finite-dimensional, connected, 1-UV Z-set in the
simply connected
 compact $Q$-manifold
 $M$.  If the inclusion $f:X\to M$ induces isomorphisms $\check H^{\ast}(M;\mathbb Z)\to
 \check H^{\ast}(X;\mathbb Z)$ on
 \v Cech cohomology with $\mathbb Z$ coefficients, then $X$ is the fixed point set of  periodic
 homeomorphisms of period $2, 3, 5, 7, 11, 13, 17$, and $19$.
\end{corollary}

\begin{proof} We use the notation of the proof of Theorem \ref{shapeobstruction}.
\v Cech and singular cohomology agree on \ANR's, and the singular cohomology groups of $M$ are
finitely generated and vanish above some dimension, so the groups
$\check H^{i}(X;\mathbb Z)$ have the same property.
By \cite{geolach}, $X$ is Shape equivalent to a finite cell complex $A$. Thus, for each $i$, $j^{\ast}: \tilde H^i
(M;\mathbb{Z})\to\tilde H^i( A_1;\mathbb{Z})$ is an isomorphism,
where $A_1 = \eta(A^Q)$ and $j:A_1\to M$ is the inclusion.  By the functoriality
of  the Universal Coefficient Theorems \cite[Theorems 5.5.10, 5.5.12]{spanier} and the 5-Lemma, for each $i$, $j_{\ast}:
\tilde H_i(A_1;\mathbb{F}_p)\to
\tilde H_i( M;\mathbb{F}_p)$ is an isomorphism for all primes $p$.

Now Corollary
\ref{QM:1conn} provides a $ \mathbb{Z}_p$ action on
$M$ with fixed point set $A_1$ and Theorem \ref{anrShape} transfers it to one with
fixed point set $X$.
\end{proof}

\section{Tool declaration}

No specialized software, or artificial intelligence tools were used in the preparation, data collection, or writing of this paper.

\end{document}